\documentclass[letter]{amsart}

\usepackage{amssymb}
\usepackage{amsfonts}
\usepackage{amsmath}
\usepackage{amsthm}
\usepackage{mathtools}
\usepackage{graphicx}
\usepackage{mathrsfs}
\usepackage{dsfont}
\usepackage{amscd}
\usepackage{multirow}
\usepackage[all]{xy}
\normalfont
\usepackage[T1]{fontenc}
\usepackage{calligra}
\usepackage{verbatim}
\usepackage{fourier}
\usepackage[usenames,dvipsnames]{color}
\usepackage[colorlinks=true,linkcolor=Blue,citecolor=Violet]{hyperref}

\usepackage{enumerate}
\usepackage{slashed}
\usepackage{microtype}
\usepackage{faktor}
\newcommand{\N}{{\mathds{N}}}
\newcommand{\Z}{{\mathds{Z}}}

\newcommand{\R}{{\mathds{R}}}
\newcommand{\C}{{\mathds{C}}}

\newcommand{\D}{{\mathfrak{D}}}
\newcommand{\A}{{\mathfrak{A}}}
\newcommand{\B}{{\mathfrak{B}}}

\newcommand{\bigslant}[2]{{\raisebox{.2em}{$#1$}\left/\raisebox{-.2em}{$#2$}\right.}}

\newcommand{\Lip}[1][L]{{\mathsf{#1}}}
\newcommand{\TLip}{{\mathsf{T}}}

\newcommand{\Hilbert}[1][H]{{\mathscr{#1}}}

\newcommand{\dpropinquity}[1]{{\mathsf{\Lambda}^\ast_{#1}}}

\newcommand{\spectralpropinquity}[1]{{\mathsf{\Lambda}^{\mathsf{spec}}_{#1}}}

\newcommand{\Kantorovich}[1]{{\mathsf{mk}_{#1}}}

\newcommand{\Haus}[1]{{\mathsf{Haus}\!\left[{#1}\right]\,}}

\newcommand{\StateSpace}{{\mathscr{S}}}

\newcommand{\MongeKant}{{Mon\-ge-Kan\-to\-ro\-vich metric}}

\newcommand{\mcc}[3]{{\mathrm{metCor}\left({#1},{#2},{#3}\right)}}

\newcommand{\qcms}{quantum compact metric space}

\newcommand{\unit}{1}

\newcommand{\sa}[1]{{\mathfrak{sa}\left({#1}\right)}}

\newcommand{\dd}[1]{{\, \mathrm{d}{#1}}}

\newcommand{\inner}[3]{{\left\langle{#1},{#2}\right\rangle_{#3}}}

\newcommand{\dom}[1]{{\operatorname*{dom}\left({#1}\right)}}

\newcommand{\diam}[2]{{\mathrm{diam}\left({#1},{#2}\right)}}

\newcommand{\norm}[2]{\left\|{#1}\right\|_{#2}}

\newcommand{\range}[1]{{\mathrm{ran}\left({#1}\right)}}
\newcommand{\grad}[2]{{\nabla_{#1}{#2}}} %{{\mathrm{grad}\left({#1}\right)}}
\newcommand{\multiplicity}[2]{{\mathrm{multiplicity}\left({#1}\middle\vert{#2}\right)}}

\newcommand{\CDN}{{\mathsf{DN}}}
\newcommand{\TDN}{{\mathsf{TN}}}

\newcommand{\worknote}[1]{}
\newcommand{\opnorm}[3]{{\left|\mkern-1.5mu\left|\mkern-1.5mu\left| {#1} \right|\mkern-1.5mu\right|\mkern-1.5mu\right|_{#3}^{#2}}}

\newcommand{\tunnelextent}[1]{{\chi\left({#1}\right)}}

\newcommand{\alg}[1]{{\mathfrak{#1}}}

\newcommand{\module}[1]{{\mathscr{#1}}}

\newcommand{\spectrum}[1]{\mathrm{Sp}\left({#1}\right)}

\newcommand{\supp}[1]{\mathrm{supp}\left({#1}\right)}
\newcommand{\ModState}[1]{\widehat{\StateSpace}}

\newcommand{\closure}[1]{\mathrm{cl}\left({#1}\right)}

\renewcommand{\geq}{\geqslant}
\renewcommand{\leq}{\leqslant}

\newcommand{\Dirac}[1][D]{{\slashed{#1}}}

\newcommand{\Cl}[1]{{\C\ell\left({#1}\right)}}

\theoremstyle{plain}
\newtheorem{theorem}{Theorem}[section]
\newtheorem*{theorem*}{Theorem}

\newtheorem{corollary}[theorem]{Corollary}

\newtheorem{lemma}[theorem]{Lemma}

\newtheorem{theorem-definition}[theorem]{Theorem-Definition}

\theoremstyle{definition}
\newtheorem{definition}[theorem]{Definition}

\newtheorem{notation}[theorem]{Notation}

\theoremstyle{remark}

\newtheorem{remark}[theorem]{Remark}

\numberwithin{equation}{section}

\usepackage[T1]{fontenc}

\begin{document}

\title[]{Continuity of almost commutative manifolds for the spectral propinquity}

\author{Fr\'{e}d\'{e}ric Latr\'{e}moli\`{e}re}
\address{Department of Mathematics \\ University of Denver \\ Denver CO 80208}
\email{frederic@math.du.edu}
\urladdr{https://fredericlatremoliere.com}

\date{\today}

\begin{abstract}
	We prove that the canonical metric spectral triples of a fixed compact connected spin manifold depend continuously, for the spectral propinquity, on the Riemannian metric endowed with the $C^1$ topology. This continuity extends to almost commutative spectral triples, allowing the Riemannian metric and the Dirac operator of a fixed finite dimensional factor to vary simultaneously. We also establish continuity for metric spectral triples arising from ergodic actions of compact connected Lie groups as the inner product on the Lie algebra varies. Our proofs use a common continuity criterion involving graph norms and commutator seminorms. As a consequence, eigenvalues in any compact interval whose endpoints avoid the limiting spectrum converge with multiplicity. These results relate classical variations of metric data to the geometry of metric spectral triples.
\end{abstract}

\maketitle

%%%%%%%%%%%%%%%%%%%%%%%%%%%%%%%%%%%%%%%%%%%%%%%%%%%%%%%%%%%%%%%%%%%%%%%%%%%%%%%%%%%%%%%%%%%%%%%%%%%%%%%%%
\tableofcontents

The \emph{spectral propinquity} \cite{Latremoliere18g,Latremoliere22,Latremoliere26b} is a distance function, up to unitary equivalence, between (possibly twisted) metric spectral triples. Spectral triples, introduced by Connes \cite{Connes}, have emerged as a preferred tool to encode geometric information over noncommutative algebras of operators on Hilbert spaces, and thus the spectral propinquity defines a topology on the space of noncommutative spectral geometry. The canonical Dirac operator of a connected closed spin Riemannian manifold is a foundational example of a metric spectral triple \cite{Connes89}. The products of such triples with finite dimensional triples, often called almost commutative geometry, have also found a valuable role in the formulation of the standard model in noncommutative geometry \cite{Latremoliere21a,Latremoliere20a,Latremoliere23a,Latremoliere24a,Latremoliere24b}. The present paper establishes in Theorems (\ref{main-thm-1}) and (\ref{main-thm-2}) continuity of these commutative and almost commutative spectral triples as the underlying Riemannian metric varies in the $C^1$ topology over a fixed background closed connected spin Riemannian manifold, while the Dirac operator on the finite dimensional factor varies in norm. We also prove, in Theorem (\ref{lie-cont-thm}), continuity for the metric spectral triples obtained from ergodic actions of compact connected Lie groups on unital C*-algebras, when the inner product on the Lie algebra varies. Such examples include quantum tori, and the underlying C*-algebra can be simple.

\medskip

All the examples in this paper, as disparate as they are --- some over possibly simple C*-algebras and thus very noncommutative, some almost commutative, and some classical --- are actually corollaries of a general criterion for continuity for the specral propinquity of families of metric spectral triples over a fixed (noncommutative) C*-algebra, which our main theorem proven here. We thus prove that continuity for the spectral propinquity follows from a continuity of the Dirac operators for a Sobolev-type norm, and a continuity of the associated Lipschitz seminorms (i.e. norms of the commutators with the Dirac operators, on an appropriate, fixed, dense *-subalgebra of the base C*-algebra). We then verify that these assumptions are met in our different examples, illustrating the flexibility of our framework.

In particular, our work extends, using a different technique, the extensive study of the dependence of the Dirac operator and its spectrum on the underlying Riemannian metric \cite{BG,Maier97,Nowaczyk13}. In fact, we borrow from this work the construction of unitaries between the Hilbert spaces of square integrable spinors for different metrics, which enables us to work on a fixed Hilbert space after conjugation by these unitaries (as the spectral propinquity is null between unitarily equivalent spectral triples). However, we actually establish our results using a different approach: instead of using Rellich theorem and using holomorphic families of operators, we involve the theory of perturbation of semigroups of operators, as fitting the spectral propinquity construction.

\medskip

In \cite{Latremoliere22}, we proved continuity of the functional calculus and of spectra for the spectral propinquity. We deduce in thi work Theorem (\ref{cont-thm}), which states that, in any compact interval whose endpoints avoid all the eigenvalus of the limit Dirac operator, the number of eigenvalues counted with multiplicity is eventually constant along a convergent sequence of spectral triples, and the ordered eigenvalues converge. We then apply this spectral stability to our examples. In the classical setting, we recover known spectral continuity results \cite{BG}; in this case, our contributions are the stronger convergence of the associated metric spectral triples, the use of a common theory also applicable unmodified to noncommutative examples, and the use of a different underlying operator theoretic technique (semigroups and Duhamel formula instead of holomorphic families and Rellich theorem). We then extend this stability results to our other, noncommutative, examples.

\medskip

Continuity for the spectral propinquity has also been established in settings where the underlying $C^*$-algebras and Hilbert spaces can vary. Examples include the convergence of fuzzy tori to quantum tori for spectral triples motivated by physics \cite{Latremoliere21a}, the convergence of fuzzy tori equipped with a natural discrete calculus to standard flat quantum tori, certain families of $C^*$-crossed products, including examples with hyperbolic actions \cite{Latremoliere26c}, approximations of fractals \cite{Latremoliere20a}, and inductive limits of twisted group $C^*$-algebras \cite{Latremoliere23a,Latremoliere23b,Latremoliere24a}, and examples of collapse \cite{Latremoliere24b}. In each of these settings, Theorem~\ref{cont-thm} gives convergence of the Dirac eigenvalues, counted with multiplicity, in every compact interval whose endpoints lie outside the spectrum of the limiting Dirac operator. The classical continuity of Dirac eigenvalues under $C^1$ variation of the Riemannian metric thus appears as a special case of spectral stability under convergence of metric spectral triples.

\medskip

We begin our presentation with a general background section on the spectral propinquity. We then establish our general theorem from which all other results here derive. Our examples are then organized in two sections: Dirac operators from ergodic actions of compact Lie groups on unital C*-algebras, and almost commutative models.

\section{Background: The Spectral Propinquity}

Noncommutative metric geometry studies notions of convergence for certain possibly noncommutative algebras which are analogues of the algebras of Lipschitz functions \cite{Weaver99} over a metric space, largely inspired by the Gromov-Hausdorff distance \cite{Hausdorff, Gromov81}. In time, a Gromov-Hausdorff type distance between metric spectral triples, which are defined below, was introduced \cite{Latremoliere18g,Latremoliere22}.  We present a rather brief summary of the construction of this distance in this section; its motivations and examples can be found in \cite{Latremoliere22}. We note that we extended this construction to certain (linearly) twisted metric spectral triples in \cite{Latremoliere26b}, though we will not include twists here as we will not need them for our examples.

\medskip

If $(X,d)$ is a compact metric space, then the natural encoding of the metric $d$ at the level of functions is given by the Lipscthiz seminorm. In turn, the Lipschitz seminorm induces by duality a metric on the state space of the C*-algebra $C(X)$ of $\C$-valued continuous functions over $X$, introduced originally by Kantorovich \cite{Kantorovich40,Kantorovich58}. The notion of a {\qcms}, introduced by Connes \cite{Connes89} and Rieffel \cite{Rieffel98a,Rieffel99} is based on this model, by identifying certain densely defined seminorms on possibly noncommutative, unital C*-algebras as analogues of the Lipscthiz seminorm. The presentation we give rests on the definition of a {\qcms} which we introduced in \cite{Latremoliere13,Latremoliere13b,Latremoliere15} which fits particularly well with spectral triples.

\begin{notation}
	If $E$ is any vector space, then the norm on $E$ is denoted by $\norm{\cdot}{E}$ unless otherwise specified. The norm of a linear endomorphism $T$ of $E$ (allowing $\infty$) is denoted by $\opnorm{T}{}{E}$. When $\A$ is a unital C*-algebra, with unit $\unit_\A$, then the state space of $\A$ is denoted by $\StateSpace(\A)$. If $a\in\A$, then $\Re a = \frac{a+a^\ast}{2}$ and $\Im a = \frac{a-a^\ast}{2i}$. The set $\{ \Re a : a \in \A\}$ of self-adjoint elements in $\A$ is denoted by $\sa{\A}$; it is a Jordan-Lie algebra for the Jordan product  $a,b\in\sa{\A}\mapsto \Re(ab)$  and the Lie product $a,b\in\sa{\A} \mapsto \Im(ab)$.
\end{notation}

\begin{definition}[\cite{Latremoliere13,Latremoliere13b}]\label{qcms-def}
	A \emph{\qcms} $(\A,\Lip)$ is an ordered pair of a unital separable C*-algebra $\A$ and a seminorm $\Lip$ defined on a dense Jordan-Lie subalgebra $\dom{\Lip}$ of $\sa{\A}$ such that $\Lip(\unit_\A) = 0$ and:
	\begin{enumerate}
		\item the {\MongeKant} $\Kantorovich{\Lip}$, defined for all $\varphi, \psi \in \StateSpace(\A)$ by
		      \begin{equation*}
			      \Kantorovich{\Lip}(\varphi,\psi) \coloneqq \sup\left\{ |\varphi(a) - \psi(a)| : a\in\dom{\Lip}, \Lip(a) \leq 1 \right\}
		      \end{equation*}
		      induces the weak* topology on $\StateSpace(\A)$,
		\item $\max\left\{ \Lip(\Re(ab)), \Lip(\Im(ab)) \right\} \leq \norm{a}{\A} \Lip(b) + \Lip(a) \norm{b}{\A}$ for all $a,b \in \dom{\Lip}$,
		\item $\{ a \in \dom{\Lip} : \Lip(a) \leq 1 \}$ is closed in $\A$.
	\end{enumerate}
\end{definition}

\begin{remark}
	We use the notation in Definition (\ref{qcms-def}). If $\Lip(a) = 0$ for $a \notin \R\unit_\A$ and $a\in\dom{\Lip}$, then there exists two states $\varphi,\psi \in \StateSpace(\A)$ such that $\varphi(a)\neq\psi(a)$, and since $\Lip(t a)=0$ for all $t > 0$,
	\begin{equation*}
		t |\varphi(a) - \psi(a)| \leq \Kantorovich{\Lip}(\varphi,\psi)
	\end{equation*}
	so $\Kantorovich{\Lip}(\varphi,\psi) = \infty$; this is not compatible with $(\StateSpace(\A),\Kantorovich{\Lip})$ being a compact space. So $\Lip(a) = 0 \implies a \in \R\unit_\A$.
\end{remark}

\begin{remark}
	Sometimes, it is helpful to generalize the Leibniz inequality in Definition (\ref{qcms-def}), (2), see \cite{Latremoliere15,Latremoliere15d}. As we work with spectral triples, this will not be necessary. Moreover, we point out that a notion of a locally compact quantum metric spaces (and associated Gromov-Hausdorff convergence) was devised in \cite{Latremoliere25b}. We will focus entirely on the unital case here.
\end{remark}

Of prime interest for us in this paper are L-seminorms induced by spectral triples in the manner suggested by Connes \cite{Connes89}. As stated in our introduction, a spectral triple is a (far reaching) abstraction of the data given by a Dirac operator. Introduced by Connes \cite{Connes89}, it is the basis for noncommutative differential geometry \cite{Connes}.
\begin{definition}[\cite{Connes89}]
	A \emph{spectral triple} $(\A,\Hilbert,\Dirac)$ is given by a unital C*-algebra $\A$ acting on a Hilbert space $\Hilbert$, and a self-adjoint operator $\Dirac$ defined on a dense subspace $\dom{\Dirac}$ of $\Hilbert$, with compact resolvent, and such that
	\begin{equation*}
		\A_{\Dirac}\coloneqq \left\{ a \in \A : a\dom{\Dirac}\subseteq\dom{\Dirac}\text{ and }[\Dirac,a] \text{ is bounded on $\dom{\Dirac}$ } \right\}
	\end{equation*}
	is dense in $\A$.
\end{definition}

If $(\A,\Hilbert,\Dirac)$ is a spectral triple, and $a \in \A_{\Dirac}$, then $[\Dirac,a]$ will still denote the unique extension by uniform continuity of $[\Dirac,a]$ to $\Hilbert$.

An important example for us in this article is also the first example in \cite{Connes89}. If $M$ is a compact closed spin manifold, if $g$ is a Riemannian metric on $M$, if $\Sigma^g M$ is the bundle of spinors over $M$ for $g$, and if $\Dirac_g$ is the Dirac operator defined by $g$ on the Hilbert space $\Gamma_{L^2}(\Sigma^g M)$ of sections of $\Sigma^g M$ which are square integrable for the volume form induced by $g$, then $(C(M),\Gamma_{L^2}(\Sigma^g M),\Dirac_g)$ is a spectral triple. If moreover, $M$ is connected, then the {\MongeKant} $\Kantorovich{\Lip_{\Dirac_g}}$ restricts to the path metric induced by $g$ on $M$ --- having identified $M$ with the space of characters of $C(M)$ endowed with the weak* topology. In particular, this spectral triple is \emph{metric} in the following natural sense.

\begin{definition}[\cite{Latremoliere18g}]\label{metric-spectral-triple-def}
	A spectral triple $(\A,\Hilbert,\Dirac)$ is \emph{metric} when $(\A,\Lip_{\Dirac})$ is a {\qcms}, where for all $a\in\A_{\Dirac}\cap\sa{\A}$, we set:
	\begin{equation*}
		\Lip_{\Dirac}(a) = \opnorm{[\Dirac,a]}{}{\Hilbert} \text.
	\end{equation*}
\end{definition}

\begin{notation}
	Keeping the notations of Definition (\ref{metric-spectral-triple-def}), the associated {\MongeKant} $\Kantorovich{\Lip_{\Dirac}}$ may be simply denoted as $\Kantorovich{\Dirac}$ and is known in this context as \emph{Connes distance} associated with the spectral triple $(\A,\Hilbert,\Dirac)$.
\end{notation}

\medskip

We refer to \cite{Latremoliere16} for an appropriate notion of \emph{Lipschitz morphisms} between {\qcms s}. For our purpose, it will suffice to define \emph{quantum isometries}. The following definition is based on the insight of Rieffel \cite{Rieffel00}, inspired by a result of McShane \cite{McShane34}.
\begin{definition}\label{quantum-isometry-def}
	Let $(\A,\Lip_\A)$ and $(\B,\Lip_\B)$ be two {\qcms s}. A \emph{quantum isometry} $\pi : (\A,\Lip_\A) \rightarrow (\B,\Lip_\B)$ is a surjective *-morphism $\pi : \A\rightarrow\B$ such that $\pi(\dom{\Lip_\A}) = \dom{\Lip_\B}$ and
	\begin{equation*}
		\forall b \in \dom{\Lip_\B} \quad \Lip_\B(b) = \inf\left\{ \Lip_\A(a) : a \in\dom{\Lip_\A}, \pi(a) = b \right\} \text.
	\end{equation*}
\end{definition}
It is easy to check that if $\pi$ is a quantum isometry, then
\begin{equation*}
	\pi^\ast : \varphi\in\StateSpace(\B) \mapsto \varphi\circ\pi \in \StateSpace(\A)
\end{equation*}
is an actual isometry from $(\StateSpace(\B),\Kantorovich{\Lip_\B})$ to $(\StateSpace(\A),\Kantorovich{\Lip_\A})$.

\medskip

To define an analogue of the Gromov-Hausdorff distance between {\qcms s}, we ``isometrically embed'' two {\qcms s} into a third and compute the Hausdorff distance between appropriate sets of states, metrized by the {\MongeKant}.  Formally, this idea is encoded in the notion of a tunnel.
\begin{notation}
	The Hausdorff distance \cite{Hausdorff} on the set of closed subsets of a metric space $(X,d)$ is denoted by $\Haus{d}$, or even $\Haus{X}$ if no confusion arises.
\end{notation}

\begin{definition}[\cite{Latremoliere14}]\label{tunnel-def}
	Let $(\A,\Lip_\A)$ and $(\B,\Lip_\B)$ be two {\qcms s}. A \emph{tunnel} $\tau = (\D,\Lip,\pi,\rho)$   from $(\A,\Lip_\A)$ to $(\B,\Lip_\B)$ is given by a {\qcms} $(\D,\Lip)$, and two quantum isometries $\pi :(\D,\Lip)\rightarrow(\A,\Lip_\A)$ and $\rho:(\D,\Lip)\rightarrow(\B,\Lip_\B)$. The \emph{extent} $\tunnelextent{\tau}$ of $\tau = (\D,\Lip,\pi,\rho)$ is:
	\begin{multline*}
		\max\left\{ \Haus{\Kantorovich{\Lip}}\left(\StateSpace(\D),\pi^\ast(\StateSpace(\A))\right),\Haus{\Kantorovich{\Lip}}\left(\StateSpace(\D),\rho^\ast(\StateSpace(\B))\right)\right\} \text.
	\end{multline*}
\end{definition}

If for any two {\qcms s} $(\A,\Lip_\A)$ and $(\B,\Lip_\B)$, we define
\begin{equation*}
	\dpropinquity{}((\A,\Lip_\A),(\B,\Lip_\B)) \coloneqq \inf\left\{ \tunnelextent{\tau} : \tau \text{ is a tunnel from }(\A,\Lip_\A)\text{ to }(\B,\Lip_\B) \right\} \text.
\end{equation*}
This metric, the \emph{propinquity}, is indeed a complete metric up to full quantum isometry on the space of {\qcms s} \cite{Latremoliere13,Latremoliere13b,Latremoliere14}, which is the foundation of our research project (e.g. \cite{Latremoliere13c,Latremoliere15,Latremoliere15d,Latremoliere16,Aguilar16a,Aguilar16b,Rieffel17a,Latremoliere17b,Aguilar18,Latremoliere21c}). We will however focus in this paper on defining our metric specifically for spectral triples, which require more ingredients, i.e. is a form of the propinquity which accounts for the additional information contained in a spectral triple besides its Connes distance.

\medskip

A spectral triple $(\A,\Hilbert,\Dirac)$ gives rise to more structure than just the underlying metric. In fact, $\Hilbert$ is an $\A$-module, and $\Dirac$ induces its graph norm on a dense subspace (analogues to a Sobolev space in this generalized setting), whose closed unit ball is compact in $\Hilbert$. We abstract this structure as our next step toward defining a distance between metric spectral triples. It will be helpful for the construction of the spectral propinquity to note that $\Hilbert$ is in fact a $\C$-Hilbert module (i.e. a Hilbert space!) carrying an action of $\A$, so that $(\Hilbert,\A,\C)$ is a $C^\ast$-correspondence, in the following sense.

A \emph{$C^\ast$-correspondence} $(\module{M},\A,\B)$ is a Hilbert $\B$-module together with a *-morphism from $\A$ to the C*-algebra of adjointable $\B$-linear operators over $\module{M}$. In particular, $\module{M}$ is an $\A$-$\B$-bimodule. We denote the $\B$-valued inner product on $\module{M}$ by $\inner{\cdot}{\cdot}{\module{M}}$, and the induced norm on $\module{M}$ by $\norm{\cdot}{\module{M}} : \omega \mapsto \sqrt{\norm{\inner{\omega}{\omega}{\module{M}}}{\B}}$.

Let $(\module{M},\A,\B)$ and $(\module{N},\D,\alg{E})$ be two $C^\ast$-correspondences. A triple $(\Pi,\pi,\rho)$ is a morphism of $C^\ast$-correspondences from $(\module{M},\A,\B)$ and $(\module{N},\D,\alg{E})$ when $\pi : \A\rightarrow\D$ and $\rho:\B\rightarrow\alg{E}$ are two unital *-morphisms, and $\Pi : \module{M}\rightarrow\module{N}$ is a linear map, such that:
\begin{enumerate}
	\item $\Pi(a \omega b) = \pi(a) \Pi(\omega) \rho(b)$ for all $a\in \A$, $b\in \B$ and $\omega\in\module{M}$,
	\item $\inner{\Pi(\omega)}{\Pi(\eta)}{\module{N}} = \rho(\inner{\omega}{\eta}{\module{M}})$ for all $\omega,\eta\in\module{M}$.
\end{enumerate}

For our purpose, a spectral triple $(\A,\Hilbert,\Dirac)$ gives rise to a $C^\ast$-correspondence with additional metric data encoded in various seminorms. Formally, if $\CDN$ is the graph norm $\xi \in \dom{\Dirac}\mapsto \norm{\xi}{\Hilbert} + \norm{\Dirac\xi}{\Hilbert}$, then
\begin{equation*}
	\mcc{\A}{\Hilbert}{\Dirac} \coloneqq (\Hilbert,\CDN,\A,\Lip_{\Dirac},\C,0)
\end{equation*}
is a special case of the following structure.
\begin{definition}[\cite{Latremoliere16c,Latremoliere18d}]\label{mcc-def}
	A \emph{metrical $C^\ast$-correspondence} $(\module{M},\CDN, \A,\Lip_\A,\B,\Lip_\B)$ is a $C^\ast$-correspondence $(\module{M},\A,\B)$ such that, moreover, $(\A,\Lip_\A)$ and $(\B,\Lip_\B)$ are {\qcms s}, and $\CDN$ is a norm on a dense $\C$-subspace $\dom{\CDN}$ of $\module{M}$ such that
	\begin{enumerate}
		\item $\{ \omega\in\module{M} : \CDN(\omega) \leq 1 \}$ is compact in $\module{M}$,
		\item $\CDN\geq\norm{\cdot}{\module{M}}$,
		\item For all $\omega,\omega'\in\dom{\CDN}$,
		      \begin{equation*}
			      \max\left\{ \Lip_\B\left(\Re\inner{\omega}{\omega'}{\module{M}}\right), \Lip_\B\left(\Im\inner{\omega}{\omega'}{\module{M}}\right)\right\} \leq 2 \CDN(\omega) \CDN(\omega') .
		      \end{equation*}
		\item $\CDN(a \omega) \leq \left(\Lip_\A(a) + \norm{a}{\A}\right)\CDN(\omega)$ for all $a\in \dom{\Lip_\A}$ and $\omega\in\dom{\CDN}$.
	\end{enumerate}
\end{definition}

\begin{remark}
	The inequalities in (3) and (4) in Definition (\ref{mcc-def}) are directly inspired from the properties of a metric connection, and are analogues of the Leibniz inequality for Lipschitz seminorms, but in this more general context.
\end{remark}

We specialize the notion of a correspondence morphism to an isometric one when working with metrical $C^\ast$-correspondences. The idea behind this definition is explained in \cite{Latremoliere16c}.

\begin{definition}[\cite{Latremoliere16c,Latremoliere18d}]
	Let
	\begin{align*}
		\mathds M_1 & \coloneqq(\module M,\CDN_{\module M},\A,\Lip_\A,\B,\Lip_\B),          \\
		\mathds M_2 & \coloneqq(\module N,\CDN_{\module N},\D,\Lip_\D,\alg E,\Lip_{\alg E})
	\end{align*}
	be two metrical $C^*$-correspondences.

	A \emph{quantum isometry} of metrical $C^\ast$-correspondences from $\mathds{M}_1$ to $\mathds{M}_2$ is a morphism of $(\Pi,\pi,\rho)$ of $C^\ast$-correspondences from $(\module{M},\A,\B)$ to $(\module{N},\D,\alg{E})$ such that
	\begin{enumerate}
		\item $\pi :(\A,\Lip_\A)\rightarrow (\D,\Lip_\D)$ and $\rho : (\B,\Lip_\B) \rightarrow (\alg{E},\Lip_{\alg{E}})$ are quantum isometries,
		\item $\Pi$ is surjective, $\Pi(\dom{\CDN_{\module{M}}}) = \dom{\CDN_{\module{N}}}$, and for all $\omega\in\dom{\CDN_{\module{N}}}$:
		      \begin{equation*}
			      \CDN_{\module{N}}(\omega) = \inf\left\{ \CDN_{\module{M}(\eta) : \eta}\in\dom{\CDN_{\module{M}}} \cap \Pi^{-1}(\{\omega\})\right\} \text.
		      \end{equation*}
	\end{enumerate}
\end{definition}

\medskip

We now define the notion of convergence for metric spectral triples. Inspired by the work of Gromov \cite{Gromov,Gromov81}, we begin by defining a form of ``isometric embedding'' for the structures above.
\begin{definition}
	Let $\mathds{A}$ and $\mathds{B}$ be two metrical $C^\ast$-correspondences. A \emph{metrical tunnel} $\tau$ from $\mathds{A}$ to $\mathds{B}$ is given as a tuple:
	\begin{equation*}
		\left(\mathds{P}, \Pi, R \right)
	\end{equation*}
	where $\mathds{P}$ is a metrical $C^\ast$-correspondence, while $\Pi :\mathds{P}\rightarrow\mathds{A}$ and $R: \mathds{P}\rightarrow\mathds{B}$ are quantum isometries.
\end{definition}

We then associate a number to a metrical tunnel. The key here is the observation that a metric tunnel contains two embeddings of {\qcms s} for which we know how to associate a number called the extent. Unfortunately, our definition requires that we expand the definition of a tunnel in all of its components for notation purposes.
\begin{definition}[\cite{Latremoliere18d}]
	The \emph{extent} of a metrical tunnel
	\begin{equation*}
		( \underbracket[1pt]{\module{M}, \CDN, \D, \Lip_\D, \alg{E},\Lip_{\alg{E}}}_{\text{metrical $C^\ast$-correspondence}}, \underbracket[1pt]{(\Pi,\pi,\varpi)}_{\substack{\text{metrical} \\ \text{quantum isometry}}}, \underbracket[1pt]{(R,\rho,\varrho)}_{\substack{\text{metrical} \\ \text{quantum isometry}}})
	\end{equation*}
	is
	\begin{equation*}
		\max\left\{ \text{extent of }(\D,\Lip_\D,\pi,\rho)\text{ and the extent of }(\alg{E},\Lip_{\alg{E}},\varpi,\varrho) \right\} \text.
	\end{equation*}
\end{definition}

When $(\A,\Hilbert_\A,\Dirac_\A)$ and $(\B,\Hilbert_\B,\Dirac_\B)$ are metric spectral triples, a \emph{tunnel} between them is a metrical tunnel between their associated correspondences
\begin{equation*}
	\mcc{\A}{\Hilbert_\A}{\Dirac_\A}\quad\text{and}\quad
	\mcc{\B}{\Hilbert_\B}{\Dirac_\B}.
\end{equation*}

We can indeed define a Gromov-Hausdorff distance between metrical C*-corr\-espon\-dences by taking the infimum of the extents of all metrical tunnels between two given metrical C*-correspondence, and this metric is complete, up to full quantum isometry, with some applications, for instance, to Heisenberg modules over quantum tori \cite{Latremoliere16c,Latremoliere18d,Latremoliere17a,Latremoliere18a}. Yet distance zero between metrical C*-correspondences constructed from metric spectral triples still does not quite give us unitary equivalence between metric spectral triples. So we strengthen our metric one last time.

\medskip

A spectral triple induces one more object on its metrical C*-correspondence: an action of $\R$ by unitaries. We want to account for this as well in our metric. We developed a covariant version of both the propinquity \cite{Latremoliere18b,Latremoliere18c,Latremoliere23b} and the modular propinquity presented above \cite{Latremoliere18g}, which is the basis for our spectral propinquity. The idea, roughly, is to use the notion of a metrical tunnel to measure how long the orbits of these actions, restricted to $[0,\infty)$, stay close to each other. This is captured most easily in our exposition in \cite{Latremoliere22} by using the following concept, which is an obvious analogue of the {\MongeKant} for metrical $C^\ast$-correspondences. Moreover, we can also use the following definitions when discussing the convergence of functional calculi for Dirac operators under our metric. It will be helpful to introduce the following notation for the analogue of the closed unit ball of a Sobolev space:
%
%\begin{definition}[\cite{Latremoliere22}]
%If $(\module{M},\CDN,\A,\Lip_\A,\B,\Lip_\B)$ is a metrical C*-correspondence, then:
%	\begin{equation*}
%		\Kantorovich{\CDN}(\omega,\omega') \coloneqq\sup\left\{ \norm{\inner{\omega - \omega'}{\eta}{\module{M}}}{\B} : \eta\in\module{B}_{\CDN} \right\}
%	\end{equation*}
%where 
%	\begin{equation*}
%		\module{B}_{\CDN} \coloneqq \left\{ \omega \in \dom{\CDN} : \CDN(\omega)\leq 1 \right\} \text.
%	\end{equation*}
%
%	For any nonempty set $J$, we extend $\Kantorovich{\CDN}$ to families indexed by $J$ by simply setting:
%\begin{equation*}
%		\Kantorovich{\CDN}((\omega_j)_{j\in J},(\omega')_{j\in J}) \coloneqq\sup_{j \in J}\Kantorovich{\CDN}(\omega_j,\omega'_j) \text.
%	\end{equation*}
%	
%\end{definition}

\begin{definition}[\cite{Latremoliere22}]\label{sep-def}
	Let $\tau$ be a metrical tunnel between the correspondences associated with two metric spectral triples $(\A_j,\Hilbert_j,\Dirac_j)$, for $j\in\{1,2\}$. Write $\module{J}$ for its bimodule, $\TDN$ for its D-norm, and $\Pi_j:\module{J}\rightarrow\Hilbert_i$ for its module quotient maps. For scalar-valued linear functionals $\varphi,\psi$ on $\module{J}$, set
	\begin{equation*}
		\Kantorovich{\CDN}(\varphi,\psi) \coloneqq\sup_{\TDN(\zeta)\leq1}|\varphi(\zeta)-\psi(\zeta)|.
	\end{equation*}
	For families of functionals indexed by a set $J$, use the supremum of these distances over $J$. If $(A_j)_{j\in J}$ and $(B_j)_{j\in J}$ are families of bounded operators on $\Hilbert_1$ and $\Hilbert_2$, respectively, their separation along $\tau$ is
	\begin{multline*}
		\operatorname{sep}_\tau((A_j)_{j\in J},(B_j)_{j\in J})
		=\Haus{\Kantorovich{\TDN}}\Bigg(
		\left\{\left(\inner{\xi}{A_j\Pi_1(\cdot)}{}\right)_{j\in J}:\CDN_1(\xi)\leq1\right\},\\
		\left\{\left(\inner{\eta}{B_j\Pi_2(\cdot)}{}\right)_{j\in J}:\CDN_2(\eta)\leq1\right\}\Bigg)\text,
	\end{multline*}
	where $\CDN_j(\xi)\coloneqq\norm{\xi}{\Hilbert_j}+\norm{\Dirac_i\xi}{\Hilbert_j}$.
\end{definition}

\begin{definition}[\cite{Latremoliere18g,Latremoliere22}]
	The \emph{spectral propinquity} between two metric spectral triples $(\A_1,\Hilbert_1,\Dirac_1)$ and $(\A_2,\Hilbert_2,\Dirac_2)$ is the minimum of $\frac{\sqrt{2}}{2}$ and
	\begin{multline*}
		\inf\Big\{\varepsilon>0:\exists\tau\text{ a metrical tunnel between these triples,}\\
		\max\big\{\tunnelextent{\tau},
		\operatorname{sep}_\tau((\exp(it\Dirac_1))_{0\leq t\leq\frac{1}{\varepsilon}},
		(\exp(it\Dirac_2))_{0\leq t\leq\frac{1}{\varepsilon}})\big\}<\varepsilon\Big\}.
	\end{multline*}
	Throughout this paper, we use the equivalent version of the spectral propinquity from \cite[Theorem 3.5]{Latremoliere22}.
\end{definition}

We will see that convergence in the sense of the spectral propinquity has interesting implications. For now, we just note the following result.
\begin{theorem}[\cite{Latremoliere18g}]
	The spectral propinquity is a metric on the space of metric spectral triples, up to unitary equivalence, i.e. for any two metric spectral triples $(\A,\Hilbert_\A,\Dirac_\A)$ and $(\B,\Hilbert_\B,\Dirac_\B)$, we have
	\begin{equation*}
		\spectralpropinquity{}((\A,\Hilbert_\A,\Dirac_\A),(\B,\Hilbert_\B,\Dirac_\B)) = 0
	\end{equation*}
	if, and only if, there exists a unitary $U : \Hilbert_\A\rightarrow \Hilbert_\B$ such that $a \in \A\mapsto U a U^\ast$ is a *-isomorphism from $\A$ onto $\B$, and
	\begin{equation*}
		U \dom{\Dirac_\A} = \dom{\Dirac_\B} \text{ and } U \Dirac_\A = \Dirac_\B U \text.
	\end{equation*}
\end{theorem}

The spectral propinquity can be applied to obtain interesting examples, such as matrix models converging to tori \cite{Latremoliere21a}, approximations of certain spectral triples over fractals \cite{Latremoliere20a}, or results on inductive limits \cite{Latremoliere23a,Latremoliere23b,Latremoliere24a,Latremoliere24b}.

Now, of prime interest for this paper is the continuity of the spectrum of the Dirac operator with respect to the spectral propinquity, which we established in \cite{Latremoliere22}. Now, using our results in \cite{Latremoliere22}, we wish to provide a form of this continuity which is directly inspired by the Riemannian geometry literature.

\section{A continuity result for the spectral propinquity}

This section establishes the two main results upon which the rest of this paper relies.

We first give a spectral consequence of convergence for the spectral propinquity, in all generality, before proving the specific criterion for convergence for the spectral propinquity  used specifically in our applications in this paper. Theorem (\ref{cont-thm}) strengthens \cite[Corollary 5.3]{Latremoliere22}: the number of eigenvalues, counted with multiplicity, in a compact interval whose endpoints avoid the limiting spectrum is eventually constant. Its proof uses the continuity of the continuous functional calculus along the tunnels witnessing spectral convergence, as established in \cite[Theorem 4.7]{Latremoliere22}.

\begin{notation}
	The spectrum of a self-adjoint operator $\Dirac$ is denoted by $\spectrum{\Dirac}$. Moreover, if $\lambda$ is an eigenvalue of $\Dirac$, then $\multiplicity{\lambda}{\Dirac}$ is the dimension of the eigenspace for $\lambda$.
\end{notation}

\begin{theorem}\label{cont-thm}
	If $(\A_n,\Hilbert_n,\Dirac_n)_{n\in\N}$ converges to a metric spectral triple $(\A_\infty,\Hilbert_\infty,\Dirac_\infty)$ for the spectral propinquity, and if $a,b\in\R\setminus\spectrum{\Dirac_\infty}$ with $a<b$, then there exists $N\in\N$ such that, if $n\geq N$, then
	\begin{equation*}
		\sum_{\lambda\in\spectrum{\Dirac_n}\cap[a,b]}\multiplicity{\lambda}{\Dirac_n}
		=\sum_{\lambda\in\spectrum{\Dirac_\infty}\cap[a,b]}\multiplicity{\lambda}{\Dirac_\infty}
		=:r<\infty.
	\end{equation*}
	If $r>0$, and if we write these eigenvalues, repeated according to multiplicity, as
	$\lambda_{n,1}\leq\cdots\leq\lambda_{n,r}$, including for $n=\infty$, then
	$\lim_{n\rightarrow\infty}\lambda_{n,j}=\lambda_{\infty,j}$ for every $j\in\{1,\ldots,r\}$; in particular,
	\begin{equation*}
		\lim_{n\to\infty}\Haus{\R}\left( \spectrum{\Dirac_n}\cap[a,b] , \spectrum{\Dirac_\infty}\cap[a,b] \right) = 0 \text.
	\end{equation*}

	If $r=0$, then eventually $\spectrum{\Dirac_n}\cap[a,b]=\emptyset$.
\end{theorem}

\begin{proof}
	We begin with a gap result. Choose $0<\varepsilon<(b-a)/2$ such that, if
	$K=[a-\varepsilon,a+\varepsilon]\cup[b-\varepsilon,b+\varepsilon]$, then $K\cap \spectrum{\Dirac_\infty}=\emptyset$.

	If $\{ n \in \N : \spectrum{\Dirac_n}\cap K \neq\emptyset\}$ is infinite, then there exists a sequence $(\lambda_n)_{n\in\N}$ in $K$, and a strictly increasing function $f : \N\to\N$, such that $\lambda_n \in \spectrum{\Dirac_{f(n)}}$ for all $n\in\N$. Since $K$ is compact, $(\lambda_n)_{n\in\N}$ has a convergent subsequence with limit in the compact $K$. By \cite[Theorem 5.2]{Latremoliere22}, its limit would belong to $\spectrum{\Dirac_\infty}$, a contradiction. Thus, there exists $N_K \in \N$ such that if $n\geq N_K$, then $\spectrum{\Dirac_n}\cap K=\emptyset$.

	Choose $p\in C_c(\R,[0,1])$ equal to $1$ on $[a+\varepsilon,b-\varepsilon]$ and equal to $0$ outside $(a,b)$. For $n\geq N_K$, including $n=\infty$,
	\begin{equation*}
		P_n\coloneqq p(\Dirac_n)
	\end{equation*}
	is precisely the spectral projection for $[a,b]$. Set $\ell=1+\max\{|a|,|b|\}$, so that if $\xi \in \range{P_n}$, as $\xi$ is a linear combination of eigenvectors for $\Dirac_n$ with eigenvalues in $[a,b]$,
	\begin{equation*}
		\CDN_n(\xi) \leq \norm{\xi}{\Hilbert_n} \ell \text.
	\end{equation*}
	Moreover,  $\CDN_n(P_n\xi)\leq\CDN_n(\xi)$ for all $\xi \in \dom{\Dirac_n}$ since $\norm{P_n\xi}{\Hilbert_n}\leq\norm{\xi}{\Hilbert_n}$ and $\norm{\Dirac_n P_n \xi}{\Hilbert_n} = \norm{P_n (\Dirac_n\xi)}{\Hilbert_n} \leq \norm{\Dirac_n\xi}{\Hilbert_n}$.

	For each $n\in\N$, let $\tau_n$ be a tunnel from $(\A_n,\Hilbert_n,\Dirac_n)$ to $(\A_\infty,\Hilbert_\infty,\Dirac_\infty)$ such that $\chi_n\coloneqq\tunnelextent{\tau_n}\rightarrow0$, as in \cite[Hypothesis 4.2]{Latremoliere22}. By \cite[Theorem 4.7]{Latremoliere22}, for every $\delta>0$ there exists $N_\delta$ such that, for $n\geq N_\delta$,
	\begin{equation*}
		\operatorname{sep}_{\tau_n}\big(p(\Dirac_n),p(\Dirac_\infty)\big)<\delta.
	\end{equation*}
	Set $s_n=\operatorname{sep}_{\tau_n}(p(\Dirac_n),p(\Dirac_\infty))+1/n$ for $n\geq1$. In particular, $s_n>0$, and $s_n<2\delta$ whenever $n\geq N_\delta$ and $n>1/\delta$.
	We write $\module{J}_n$ for the bimodule of $\tau_n$, $\TDN_n$ for its norm, and $\Pi_{n,j}$ for its quotient map from $\module{J}_n$ onto $\Hilbert_j$, where $j \in \{n,\infty\}$.

	Let $m \in \N\setminus\{0\}$, and let $\delta>0$ be such that
	\begin{equation*}
		16m\ell^2\delta<1\text.
	\end{equation*}
	Let $N_\delta$ be as above. Since $\lim_{n\to\infty}\chi_n=0$, there exists $N_\chi$ such that, if $n\geq N_\chi$, then $\chi_n<\delta$. Let $N=\max\{N_K,N_\delta,N_\chi,\lfloor 1/\delta\rfloor+1\}$. Fix $n\geq N$.

	For $\{j,k\} = \{n,\infty\}$,  $\xi\in\Hilbert_j$ and $\eta\in\Hilbert_k$, we write
	\begin{equation*}
		q_n(\xi,\eta)=\sup_{\TDN_n(\omega)\leq 1}
		\left|\inner{\xi}{\Pi_j(\omega)}{}-\inner{\eta}{\Pi_k(\omega)}{}\right|.
	\end{equation*}
	Let $\{\xi_1,\ldots,\xi_m\}$ be an orthonormal family of vectors in $\range{P_j}$. Since $p$ is real-valued, by definition of the separation, there exist vectors $\tilde{\eta}_1, \ldots,\tilde{\eta}_m \in\dom{\Dirac_k}$ such that, for all $x\in\{1,\ldots,m\}$, we have both  $\CDN_k(\tilde{\eta}_x)\leq\ell$ and
	\begin{equation*}
		\sup_{\TDN_n(\omega)\leq1}\left| \inner{p(\Dirac_j)\xi_x}{\Pi_j(\omega)}{\Hilbert_j}-\inner{p(\Dirac_k)\tilde{\eta}_x}{\Pi_k(\omega)}{\Hilbert_k}\right|<\ell s_n.
	\end{equation*}
	By assumption, $p(\Dirac_j)\xi_x=\xi_x$ for all $x\in\{1,\ldots,m\}$. Set $\eta_x=p(\Dirac_k)\tilde{\eta}_x$. Then, for all $x \in \{1,\ldots,m\}$, we conclude
	\begin{equation}\label{q-n-sep-eq}
		\CDN_k(\eta_x)\leq\ell,\qquad q_n(\xi_x,\eta_x)<\ell s_n.
	\end{equation}
	To conclude $\{\eta_1,\ldots,\eta_m\}$ is linearly independent, we now compare its Gram matrix to the Gram matrix of $\{\xi_1,\ldots,\xi_m\}$.

	Since $\Pi_j$ is (part of a) quantum isometry, and by compactness of the $\TDN_n$-unit ball, there exists $\omega_x\in\module{J}_n$ with $\Pi_j(\omega_x)=\xi_x$ and $\TDN_n(\omega_x)\leq\ell$ for each $x\in\{1,\ldots,m\}$. Set $\zeta_x\coloneqq\Pi_k(\omega_x)$ for all $x\in\{1,\ldots,m\}$. The two quotient maps from the right coefficient algebra of the C*-correspondence $\module{J}_n$ to $\C$ are states (they are in fact unital *-morphisms to $\C$, hence states). By definition of the extent $\chi_n$ of $\tau_n$, the Monge--Kantorovich distance between these two states is at most $\chi_n$. The inner Leibniz inequality, applied separately to real and imaginary parts, therefore gives, for $\omega,\omega'\in\dom{\TDN_n}$,
	\begin{equation}\label{this-eq}
		\left|\inner{\Pi_j(\omega)}{\Pi_j(\omega')}{\Hilbert_j}-\inner{\Pi_k(\omega)}{\Pi_k(\omega')}{\Hilbert_k}\right| \leq 4 \chi_n\TDN_n(\omega)\TDN_n(\omega')\text.
	\end{equation}
	Consequently, using Eq. \eqref{this-eq} with $\omega=\omega_x$
	\begin{equation}\label{q-n-eq}
		q_n(\xi_x,\zeta_x)\leq4\ell\chi_n
	\end{equation}
	and for all $x,y \in \{1,\ldots,m\}$, using $\omega_x$ and $\omega_y$ in Eq. \eqref{this-eq} as $\omega$ and $\omega'$, and since $\inner{\Pi_j(\omega_x)}{\Pi_j(\omega_y)}{\Hilbert_j} = \inner{\xi_x}{\xi_y}{\Hilbert_j} = 1_{x=y}$, we conclude
	\begin{equation}\label{lift-gram-eq}
		|\inner{\zeta_x}{\zeta_y}{\Hilbert_k}-1_{x=y}|\leq4\ell^2\chi_n\text,
	\end{equation}
	where $1_{x=y}$ is the Kronecker symbol for $x,y$.

	Let $\theta\in\dom{\Dirac_k}$ satisfy $\CDN_k(\theta)\leq\ell$.
	Since $\Pi_k$ is (part of) a quantum isometry, and since the unit ball of $\TDN_n$ is compact,  there exists $\omega\in\dom{\TDN_n}$ such that
	$\Pi_k(\omega)=\theta$ and $\TDN_n(\omega)\leq\ell$.
	Applying the definition of $q_n$ to $\ell^{-1}\omega$, we obtain for all $x \in \{1,\ldots,m\}$,
	\begin{equation}\label{lift-eq}
		\begin{split}
			\left|\inner{\zeta_x-\eta_x}{\theta}{\Hilbert_k}\right|
			 & \leq\left|\inner{\zeta_x}{\Pi_k(\omega)}{}-\inner{\xi_x}{\Pi_j(\omega)}{\Hilbert_j}\right|             \\
			 & \quad+ \left|\inner{\xi_x}{\Pi_j(\omega)}{\Hilbert_j}-\inner{\eta_x}{\Pi_k(\omega)}{\Hilbert_k}\right| \\
			 & \leq \ell\big(q_n(\xi_x,\zeta_x)+q_n(\xi_x,\eta_x)\big)                                                \\
			 & \leq \ell^2(s_n+4\chi_n) \text{ by Eqs. \eqref{q-n-sep-eq} and \eqref{q-n-eq}.}
		\end{split}
	\end{equation}

	For each $x\in\{1,\ldots,m\}$, we have $\CDN_k(\eta_x)\leq\ell$
	and $\CDN_k(\zeta_x)\leq\TDN_n(\omega_x)\leq\ell$.
	Thus, for all $x,y\in\{1,\ldots,m\}$, applying
	\eqref{lift-eq} with $\theta=\eta_y$ and $\theta=\zeta_x$, respectively, in each of the following terms, we obtain
	\begin{align*}
		 & \left|\inner{\eta_x}{\eta_y}{\Hilbert_k}
		-\inner{\zeta_x}{\zeta_y}{\Hilbert_k}\right|                                                                                            \\
		 & \qquad\leq
		\left|\underbracket{\inner{\eta_x-\zeta_x}{\eta_y}{\Hilbert_k}}_{\theta=\eta_y \text{ in Eq. \eqref{lift-eq}}}\right|
		+\left|\underbracket{\inner{\zeta_x}{\eta_y-\zeta_y}{\Hilbert_k}}_{\theta=\zeta_x \text{ in Eq. \eqref{lift-eq}, and conjugate}}\right| \\
		 & \qquad\leq 2\ell^2(s_n+4\chi_n).
	\end{align*}
	Combining this with \eqref{lift-gram-eq} gives
	\begin{align*}
		\left|\inner{\eta_x}{\eta_y}{\Hilbert_k}-\delta_{xy}\right|
		 & \leq \left|\inner{\eta_x}{\eta_y}{\Hilbert_k}-\inner{\zeta_x}{\zeta_y}{\Hilbert_k}\right| + \left|\inner{\zeta_x}{\zeta_y}{\Hilbert_k} - \delta_{xy}\right| \\
		 & \leq 2\ell^2(s_n+4\chi_n)+4\ell^2\chi_n                                                                                                                     \\
		 & =2\ell^2s_n+12\ell^2\chi_n                                                                                                                                  \\
		 & \leq 16\ell^2\delta \text{ since }s_n < 2\delta \text{ and }\chi_n\leq \delta \text,                                                                        \\
		 & < \frac{1}{m} \text.
	\end{align*}

	Thus, the Gram matrix of $\{\eta_1,\ldots,\eta_m\}$ is invertible, so $\{\eta_1,\ldots,\eta_m\}$ is linearly independent. Equivalently, for all $c_1,\ldots,c_m\in\C$,
	\begin{align*}
		\norm{\sum_{x=1}^m c_x\eta_x}{\Hilbert_k}^2
		 & \geq \sum_{x=1}^m |c_x|^2
		-16\ell^2\delta\left(\sum_{x=1}^m |c_x|\right)^2 \\
		 & \geq (1-16m\ell^2\delta)\sum_{x=1}^m |c_x|^2,
	\end{align*}
	where the last inequality follows from Cauchy--Schwarz. Thus $\eta_1,\ldots,\eta_m$ are linearly independent since $16m\ell^2\delta <1$ by construction.

	Set $r=\dim\range{P_\infty}$, which is finite since $\Dirac_\infty$ has compact resolvent. Choose $\delta>0$ such that $16(r+1)\ell^2\delta<1$, and let $N$ be as above for this choice of $\delta$.

	Fix $n\geq N$. If $r>0$, let $\{\xi_1,\ldots,\xi_r\}$ be an orthonormal basis of $\range{P_\infty}$. Applying the preceding estimate with $(j,k)=(\infty,n)$ and $m=r$, we obtain $r$ linearly independent vectors in $\range{P_n}$, since $16r\ell^2\delta<1$. Thus $\dim\range{P_n}\geq r$, which also holds if $r=0$.

	Suppose now that $\dim\range{P_n}\geq r+1$, and choose an orthonormal family of $r+1$ vectors in $\range{P_n}$. Applying the same estimate with $(j,k)=(n,\infty)$ and $m=r+1$, we obtain $r+1$ linearly independent vectors in $\range{P_\infty}$, since $16(r+1)\ell^2\delta<1$. This contradicts $\dim\range{P_\infty}=r$. We have therefore proved that
	\begin{equation}\label{spectral-window-rank-eq}
		\dim\range{P_n}=\dim\range{P_\infty}
		\qquad\text{for all }n\geq N.
	\end{equation}
	If $r=0$, it follows that $\spectrum{\Dirac_n}\cap[a,b]=\emptyset$ for all $n\geq N$.

	Assume henceforth that $r>0$. For $j\in\{n,\infty\}$ with $n\geq N$, write $\mu_{j,1}\leq\cdots\leq\mu_{j,r}$ for the eigenvalues of $\Dirac_j$ in $[a,b]$, repeated according to multiplicity. For each $c\in(a,b)$, set
	\begin{equation*}
		\nu_j(c)\coloneqq\#\{t\in\{1,\ldots,r\}:\mu_{j,t}\leq c\}.
	\end{equation*}
	Thus $\nu_j(c)$ is the rank of the spectral projection of $\Dirac_j$ for $[a,c]$. If $c\notin\spectrum{\Dirac_\infty}$, then both endpoints of $[a,c]$ lie outside $\spectrum{\Dirac_\infty}$. Applying Eq.~\eqref{spectral-window-rank-eq} with $[a,c]$ in place of $[a,b]$, we obtain $N(c)\geq N$ such that
	\begin{equation}\label{spectral-window-count-eq}
		\nu_n(c)=\nu_\infty(c)
		\qquad\text{for all }n\geq N(c).
	\end{equation}

	Let $\varepsilon>0$, and let $\lambda$ be a distinct eigenvalue of $\Dirac_\infty$ in $[a,b]$. Denote by $u_\lambda$ and $v_\lambda$ the first and last indices at which $\lambda$ occurs in the ordered list for $\Dirac_\infty$. Since $\lambda\in(a,b)$ and $\lambda$ is isolated in $\spectrum{\Dirac_\infty}$, choose $0<\rho_\lambda<\varepsilon$ such that
	\begin{equation*}
		[\lambda-\rho_\lambda,\lambda+\rho_\lambda]\subseteq(a,b)
		\quad\text{and}\quad
		[\lambda-\rho_\lambda,\lambda+\rho_\lambda]
		\cap\spectrum{\Dirac_\infty}=\{\lambda\}.
	\end{equation*}
	Set $c_\lambda^-=\lambda-\rho_\lambda$ and $c_\lambda^+=\lambda+\rho_\lambda$. The eigenvalues of $\Dirac_\infty$ in $[a,c_\lambda^-]$ occupy precisely the first $u_\lambda-1$ positions, while those in $[a,c_\lambda^+]$ occupy the first $v_\lambda$ positions. Hence, by Eq.~\eqref{spectral-window-count-eq}, for every $n\geq\max\{N(c_\lambda^-),N(c_\lambda^+)\}$,
	\begin{equation*}
		\nu_n(c_\lambda^-)=u_\lambda-1
		\qquad\text{and}\qquad
		\nu_n(c_\lambda^+)=v_\lambda.
	\end{equation*}
	It follows that, for every $t\in\{u_\lambda,\ldots,v_\lambda\}$,
	\begin{equation*}
		c_\lambda^-<\mu_{n,t}\leq c_\lambda^+
		\quad\text{and hence}\quad
		|\mu_{n,t}-\mu_{\infty,t}|=|\mu_{n,t}-\lambda|
		\leq\rho_\lambda<\varepsilon.
	\end{equation*}

	Let $N_\varepsilon$ be the maximum of the integers $N(c_\lambda^-)$ and $N(c_\lambda^+)$ over the finitely many distinct eigenvalues $\lambda$ of $\Dirac_\infty$ in $[a,b]$. Since each $t\in\{1,\ldots,r\}$ belongs to the block of indices for $\lambda=\mu_{\infty,t}$, we conclude that
	\begin{equation*}
		\max_{1\leq t\leq r}|\mu_{n,t}-\mu_{\infty,t}|<\varepsilon
		\qquad\text{for all }n\geq N_\varepsilon.
	\end{equation*}
	Finally, set $S_j=\spectrum{\Dirac_j}\cap[a,b]$. For $\{j,k\}=\{n,\infty\}$, each point of $S_j$ has the form $\mu_{j,t}$ for some $t\in\{1,\ldots,r\}$, and $\mu_{k,t}\in S_k$. Consequently, for all $n\geq N_\varepsilon$,
	\begin{equation*}
		\sup_{\lambda\in S_j}\inf_{\mu\in S_k}|\lambda-\mu|
		\leq\max_{1\leq t\leq r}|\mu_{j,t}-\mu_{k,t}|<\varepsilon.
	\end{equation*}
	Taking the maximum over the two choices of $(j,k)$ gives the asserted Hausdorff convergence.
\end{proof}

\bigskip
We now prove that some natural conditions on a family of metric spectral triples imply convergence for the propinquity. The hypotheses of this lemma were motivated by \cite[Theorem 2.1]{Nowaczyk13}.

\begin{lemma}\label{main-lemma}
	Let $\A$ be a unital separable C*-algebra acting on a separable
	Hilbert space $\Hilbert$, and let $\Omega$ be a topological space.
	Let $E$ be a dense subspace of $\Hilbert$. We assume that, for
	each $\sigma\in\Omega$, we are given a metric spectral triple
	$(\A,\Hilbert,\Dirac_\sigma)$ with $\dom{\Dirac_\sigma}=E$.
	For $\xi\in E$, set
	\begin{equation*}
		\CDN_\sigma(\xi)
		\coloneqq
		\norm{\xi}{\Hilbert}
		+\norm{\Dirac_\sigma\xi}{\Hilbert}.
	\end{equation*}
	For $a\in\A$, set
	\begin{equation*}
		\Lip_\sigma(a)
		\coloneqq
		\opnorm{[\Dirac_\sigma,a]}{}{\Hilbert}
	\end{equation*}
	if $aE\subseteq E$ and $[\Dirac_\sigma,a]$ is bounded on $E$,
	and set $\Lip_\sigma(a)=\infty$ otherwise. We use the restriction
	of $\Lip_\sigma$ to $\sa{\A}$ as the associated L-seminorm.

	Fix $\omega\in\Omega$. Suppose that:
	\begin{enumerate}
		\item for every $\varepsilon>0$, there exists a neighborhood
		      $\Omega_\omega$ of $\omega$ such that, for all
		      $\sigma\in\Omega_\omega$,
		      \begin{equation*}
			      \dom{\Lip_\sigma}=\dom{\Lip_\omega},
		      \end{equation*}
		      and, for all $a\in\dom{\Lip_\omega}$,
		      \begin{equation*}
			      \left|\Lip_\omega(a)-\Lip_\sigma(a)\right|
			      \leq\varepsilon\Lip_\omega(a);
		      \end{equation*}
		\item for every $\varepsilon>0$, there exists a neighborhood
		      $\Omega_\omega$ of $\omega$ such that, for all
		      $\sigma\in\Omega_\omega$,
		      \begin{equation*}
			      \sup_{\substack{\xi\in E\\\CDN_\omega(\xi)\leq1}}
			      \norm{(\Dirac_\sigma-\Dirac_\omega)\xi}{\Hilbert}
			      <\varepsilon.
		      \end{equation*}
	\end{enumerate}
	Then
	\begin{equation*}
		\lim_{\substack{\sigma\rightarrow\omega\\\sigma\in\Omega}}
		\spectralpropinquity{}
		\left(
		(\A,\Hilbert,\Dirac_\sigma),
		(\A,\Hilbert,\Dirac_\omega)
		\right)
		=0.
	\end{equation*}
\end{lemma}

\begin{proof}
	Set
	\begin{equation*}
		K\coloneqq
		2\diam{\StateSpace(\A)}{\Kantorovich{\Lip_\omega}}+1.
	\end{equation*}
	Fix
	\begin{equation*}
		0<\varepsilon<
		\min\left\{1,\frac{1}{K+1}\right\},
		\qquad
		\delta\coloneqq\varepsilon^2.
	\end{equation*}
	In particular, $\delta<1/2$ and $K\delta<1$.

	By hypothesis (1), there exists a neighborhood $\Omega_0$
	of $\omega$ such that, for every $\sigma\in\Omega_0$, the
	domains of $\Lip_\omega$ and $\Lip_\sigma$ coincide and
	\begin{equation*}
		(1-\delta)\Lip_\omega(a)
		\leq\Lip_\sigma(a)
		\leq(1+\delta)\Lip_\omega(a)
	\end{equation*}
	for all $a$ in their common domain. Consequently,
	\begin{equation*}
		(1-\delta)\Kantorovich{\Lip_\sigma}
		\leq\Kantorovich{\Lip_\omega}
		\leq(1+\delta)\Kantorovich{\Lip_\sigma}.
	\end{equation*}
	Thus
	\begin{equation}\label{qdiam-eq}
		\diam{\StateSpace(\A)}{\Kantorovich{\Lip_\sigma}}
		\leq
		2\diam{\StateSpace(\A)}{\Kantorovich{\Lip_\omega}}
		<K.
	\end{equation}

	By hypothesis (2), there exists a neighborhood $\Omega_1$
	of $\omega$ such that, for every $\sigma\in\Omega_1$ and
	$\xi\in E$,
	\begin{equation}\label{relative-Dirac-estimate}
		\norm{(\Dirac_\sigma-\Dirac_\omega)\xi}{\Hilbert}
		\leq K\delta\CDN_\omega(\xi).
	\end{equation}
	Set $\Omega_2\coloneqq\Omega_0\cap\Omega_1$, and fix
	$\sigma\in\Omega_2$. It follows that
	\begin{equation}\label{graph-norm-comparison}
		(1-K\delta)\CDN_\omega(\xi)
		\leq\CDN_\sigma(\xi)
		\leq(1+K\delta)\CDN_\omega(\xi)
		\qquad(\xi\in E).
	\end{equation}

	Let $\D\coloneqq\A\oplus\A$. For self-adjoint $a,b$ in the
	common domain of our L-seminorms, define
	\begin{equation*}
		\TLip_\sigma(a,b)
		\coloneqq
		\max\left\{
		\Lip_\omega(a),
		\Lip_\sigma(b),
		\frac{\norm{a-b}{\A}}{K\delta}
		\right\}.
	\end{equation*}
	This defines a lower semicontinuous Leibniz L-seminorm on
	$\D$. Indeed, the Leibniz inequality follows from those for
	$\Lip_\omega$ and $\Lip_\sigma$, together with
	\begin{equation*}
		\norm{ac-bd}{\A}
		\leq
		\norm{a}{\A}\norm{c-d}{\A}
		+\norm{a-b}{\A}\norm{d}{\A}.
	\end{equation*}
	For metricity, fix $\mu\in\StateSpace(\A)$ and center the
	$\TLip_\sigma$-unit ball by the state $(a,b)\mapsto\mu(a)$.
	Its first coordinates form a totally bounded set. Its second
	coordinates satisfy $\Lip_\sigma(b)\leq1$ and
	$|\mu(b)|\leq K\delta$, so they also form a totally bounded
	set. The usual total boundedness characterization of
	L-seminorms therefore applies.

	Denote the coordinate projections by
	\begin{equation*}
		\pi_\omega(a,b)=a,
		\qquad
		\pi_\sigma(a,b)=b.
	\end{equation*}
	We verify that these are quantum isometries. Suppose first
	that $a=a^\ast$ and $\Lip_\sigma(a)=1$, and set
	\begin{equation*}
		b\coloneqq(1-\delta)a+\delta\mu(a)1_\A.
	\end{equation*}
	Then
	\begin{equation*}
		\Lip_\omega(b)
		=(1-\delta)\Lip_\omega(a)
		\leq1,
	\end{equation*}
	and, by \eqref{qdiam-eq},
	\begin{equation*}
		\norm{a-b}{\A}
		=\delta\norm{a-\mu(a)1_\A}{\A}
		\leq K\delta.
	\end{equation*}
	Hence $\TLip_\sigma(b,a)=1$.

	Conversely, if $a=a^\ast$ and $\Lip_\omega(a)=1$, set
	\begin{equation*}
		b\coloneqq
		\frac{a+\delta\mu(a)1_\A}{1+\delta}.
	\end{equation*}
	Then $\Lip_\sigma(b)\leq1$ and
	\begin{equation*}
		\norm{a-b}{\A}
		=
		\frac{\delta}{1+\delta}
		\norm{a-\mu(a)1_\A}{\A}
		\leq K\delta,
	\end{equation*}
	so $\TLip_\sigma(a,b)=1$. Homogeneity, together with the
	diagonal lifts of scalar elements, proves the quotient
	identities for both projections.

	Thus $(\D,\TLip_\sigma,\pi_\omega,\pi_\sigma)$ is a tunnel.
	Its extent is at most $K\delta$. Indeed, every state of
	$\D$ has the form
	\begin{equation*}
		\varphi
		=s\varphi_1\circ\pi_\omega
		+(1-s)\varphi_2\circ\pi_\sigma
	\end{equation*}
	for $s\in[0,1]$ and $\varphi_1,\varphi_2\in\StateSpace(\A)$.
	Setting $\psi=s\varphi_1+(1-s)\varphi_2$, we obtain
	\begin{equation*}
		\Kantorovich{\TLip_\sigma}
		(\varphi,\psi\circ\pi_\omega)
		\leq K\delta,
		\qquad
		\Kantorovich{\TLip_\sigma}
		(\varphi,\psi\circ\pi_\sigma)
		\leq K\delta.
	\end{equation*}

	We next regard $\Hilbert\oplus\Hilbert$ as a Hilbert
	$\C^2$-module, with coordinatewise inner product and the
	coordinatewise action of $\D$. For $\xi,\eta\in E$, define
	\begin{equation*}
		\TDN(\xi,\eta)
		\coloneqq
		\max\left\{
		\CDN_\omega(\xi),
		\CDN_\sigma(\eta),
		\frac{\norm{\xi-\eta}{\Hilbert}}{K\delta}
		\right\}.
	\end{equation*}
	Its unit ball is compact: it is a closed subset of the
	product of the compact unit balls of $\CDN_\omega$ and
	$\CDN_\sigma$.

	For $(z,w)\in\C^2$, set
	\begin{equation*}
		\Lip[Q](z,w)\coloneqq\frac{|z-w|}{K\delta},
	\end{equation*}
	using its restriction to the self-adjoint part as an
	L-seminorm. The coordinate projections
	\begin{equation*}
		j_\omega(z,w)=z,
		\qquad
		j_\sigma(z,w)=w
	\end{equation*}
	define a tunnel from $\C$ to $\C$ of extent $K\delta$.

	The module estimates follow from the corresponding
	estimates for $\CDN_\omega$ and $\CDN_\sigma$. In particular,
	\begin{equation*}
		\TDN(a\xi,b\eta)
		\leq
		\left(\norm{(a,b)}{\D}+\TLip_\sigma(a,b)\right)
		\TDN(\xi,\eta)
	\end{equation*}
	for self-adjoint $(a,b)$ in the domain of $\TLip_\sigma$.
	For the inner product estimate, observe that
	\begin{align*}
		\left|
		\inner{\xi}{\zeta}{\Hilbert}
		-\inner{\eta}{\zeta'}{\Hilbert}
		\right|
		 & \leq
		\norm{\xi-\eta}{\Hilbert}\norm{\zeta}{\Hilbert}
		+\norm{\eta}{\Hilbert}\norm{\zeta-\zeta'}{\Hilbert} \\
		 & \leq
		2K\delta\TDN(\xi,\eta)\TDN(\zeta,\zeta').
	\end{align*}
	Thus the required estimates for the real and imaginary
	parts of the $\C^2$-valued inner product hold.

	Define
	\begin{equation*}
		\Pi_\omega(\xi,\eta)=\xi,
		\qquad
		\Pi_\sigma(\xi,\eta)=\eta.
	\end{equation*}
	If $\CDN_\omega(\xi)=1$, then
	\eqref{graph-norm-comparison} gives
	\begin{equation*}
		\TDN\left(\xi,\frac{\xi}{1+K\delta}\right)=1.
	\end{equation*}
	If $\CDN_\sigma(\eta)=1$, then the same comparison gives
	\begin{equation*}
		\TDN\big((1-K\delta)\eta,\eta\big)=1.
	\end{equation*}
	Together with homogeneity and the inequalities
	$\CDN_\omega\circ\Pi_\omega\leq\TDN$ and
	$\CDN_\sigma\circ\Pi_\sigma\leq\TDN$, these identities prove
	that the coordinate maps are quantum isometries of
	metrical C*-correspondences.

	We have therefore constructed a metrical tunnel
	\begin{equation*}
		\tau\coloneqq
		\left(
		\Hilbert\oplus\Hilbert,\TDN,
		\D,\TLip_\sigma,\C^2,\Lip[Q],
		(\Pi_\omega,\pi_\omega,j_\omega),
		(\Pi_\sigma,\pi_\sigma,j_\sigma)
		\right)
	\end{equation*}
	with $\tunnelextent{\tau}\leq K\delta$.

	It remains to control the unitary groups generated by our
	Dirac operators. For $\xi\in E$ and $t\geq0$,
	differentiating
	\begin{equation*}
		s\longmapsto
		\exp(i(t-s)\Dirac_\sigma)\exp(is\Dirac_\omega)\xi
	\end{equation*}
	and integrating gives
	\begin{align*}
		 & \left(
		\exp(it\Dirac_\sigma)-\exp(it\Dirac_\omega)
		\right)\xi \\
		 & \qquad=
		i\int_0^t
		\exp(i(t-s)\Dirac_\sigma)
		(\Dirac_\sigma-\Dirac_\omega)
		\exp(is\Dirac_\omega)\xi\,\mathrm{d}s.
	\end{align*}
	The differentiation is justified by the common domain $E$,
	the equivalence of the graph norms, and
	\eqref{relative-Dirac-estimate}. Since
	$\exp(is\Dirac_\omega)$ preserves $\CDN_\omega$, we obtain
	\begin{equation*}
		\norm{
			\left(
			\exp(it\Dirac_\sigma)-\exp(it\Dirac_\omega)
			\right)\xi
		}{\Hilbert}
		\leq K\delta t\,\CDN_\omega(\xi).
	\end{equation*}

	The same estimate holds with $t$ replaced by $-t$, by applying the preceding calculation to $-\Dirac_\sigma$ and $-\Dirac_\omega$. Since $\inner{\xi}{\exp(it\Dirac)\zeta}{}=\inner{\exp(-it\Dirac)\xi}{\zeta}{}$, this is the estimate for the pulled-back orbit functionals in Definition (\ref{sep-def}).

	Suppose that $\TDN(\xi,\eta)\leq1$ and
	$\TDN(\zeta,\zeta')\leq1$. Then
	\begin{align*}
		 & \left|
		\inner{\exp(-it\Dirac_\sigma)\eta}{\zeta'}{\Hilbert}
		-
		\inner{\exp(-it\Dirac_\omega)\xi}{\zeta}{\Hilbert}
		\right|                                           \\
		 & \quad\leq
		\norm{\eta-\xi}{\Hilbert}\norm{\zeta'}{\Hilbert}  \\
		 & \qquad+
		\norm{
			\left(
			\exp(-it\Dirac_\sigma)-\exp(-it\Dirac_\omega)
			\right)\xi
		}{\Hilbert}\norm{\zeta}{\Hilbert}     \\
		 & \qquad+
		\norm{\xi}{\Hilbert}\norm{\zeta'-\zeta}{\Hilbert} \\
		 & \quad\leq 2K\delta+K\delta t.
	\end{align*}
	For $0\leq t\leq(5K\varepsilon)^{-1}$, we conclude
	\begin{equation*}
		\left| \inner{\exp(-it\Dirac_\sigma)\eta}{\zeta'}{\Hilbert} - \inner{\exp(-it\Dirac_\omega)\xi}{\zeta}{\Hilbert}\right| \leq 2K\varepsilon^2+\frac{\varepsilon}{5} <5K\varepsilon \text.
	\end{equation*}
	Let $I_\varepsilon=[0,(5K\varepsilon)^{-1}]$. If
	$\CDN_\omega(\xi)\leq1$, set $\eta=\xi/(1+K\delta)$. Then
	$\TDN(\xi,\eta)\leq1$, so $\CDN_\sigma(\eta)\leq1$.
	Conversely, if $\CDN_\sigma(\eta)\leq1$, set
	$\xi=(1-K\delta)\eta$. Again, $\TDN(\xi,\eta)\leq1$, so
	$\CDN_\omega(\xi)\leq1$. In both cases the choice of the
	second vector is independent of $t$.

	Consequently, for each $(j,k)\in\{(\omega,\sigma),(\sigma,\omega)\}$,
	the preceding estimate gives
	\begin{align*}
		 & \sup_{\CDN_j(\xi)\leq1}
		\inf_{\CDN_k(\eta)\leq1}
		\sup_{t\in I_\varepsilon}
		\sup_{\TDN(\zeta,\zeta')\leq1}
		\left|
		\inner{\xi}
		{\exp(it\Dirac_j)\Pi_j(\zeta,\zeta')}{\Hilbert}
		-
		\inner{\eta}
		{\exp(it\Dirac_k)\Pi_k(\zeta,\zeta')}{\Hilbert}
		\right|                    \\
		 & \qquad\leq
		\sup_{t\in I_\varepsilon}\big(2K\delta+K\delta t\big)
		=2K\varepsilon^2+\frac{\varepsilon}{5}.
	\end{align*}
	By Definition~\ref{sep-def}, the maximum over these two choices
	of $(j,k)$ is the separation of the two unitary families on
	$I_\varepsilon$. Thus this separation is at most
	$2K\varepsilon^2+\varepsilon/5<5K\varepsilon$.

	Since also
	$\tunnelextent{\tau}\leq K\varepsilon^2<5K\varepsilon$,
	the definition of the spectral propinquity yields
	\begin{equation*}
		\spectralpropinquity{}
		\left(
		(\A,\Hilbert,\Dirac_\sigma),
		(\A,\Hilbert,\Dirac_\omega)
		\right)
		\leq5K\varepsilon.
	\end{equation*}
	This holds for every $\sigma\in\Omega_2$. As $\varepsilon$
	can be chosen arbitrarily small, the desired continuity
	follows.
\end{proof}
We now turn to applications for the results in this section.

\section{Metric Spectral Triples from Compact Lie group actions}

We now present a noncommutative example of application of Lemma (\ref{main-lemma}). Quantum tori are foundational examples for noncommutative geometry \cite{Connes80,Rieffel90}, and their metric properties were explored in detail by Rieffel \cite{Rieffel98a,Rieffel00}, then by the author \cite{Latremoliere13c,Latremoliere15c,Latremoliere17a,Latremoliere18a,Latremoliere21a}. The construction in \cite{Rieffel98a} depends on an inner product on the Lie algebra of the acting group. We prove here that the associated metric spectral triples vary continuously for the spectral propinquity.

\medskip

In this section, we fix a unital C*-algebra $\A$ and a strongly continuous action $\alpha$ of a compact Lie group $G$ on $\A$, whose identity component $G^\circ$ acts ergodically, that is,
\begin{equation*}
	\{a\in\A:\forall u\in G^\circ\quad\alpha^u(a)=a\}=\C\unit_\A.
\end{equation*}
We use its unique invariant state $\eta$, characterized by $\eta(a)\unit_\A=\int_G\alpha^u(a)\,\mathrm{d}u$, where Haar measure is normalized. This state is faithful: if $a\geq0$ has zero average, then, for every state $\varphi$, the continuous nonnegative function  $u\mapsto\varphi(\alpha^u(a))$ has zero integral, and hence $\varphi(a)=0$.

Let $\mathfrak{G}$ be the Lie algebra of $G$. We fix an orientation of $\mathfrak{G}$ once and for all. The smooth elements form a dense *-subalgebra $\A_\infty$ of $\A$ \cite{Bratteli79}. For $a\in\A_\infty$ and $X\in\mathfrak{G}$, set
\begin{equation*}
	d_X a=\lim_{t\rightarrow0}\frac{\alpha^{\exp(tX)}(a)-a}{t}.
\end{equation*}
\medskip

Let now $h$ be any inner product on $\mathfrak{G}'$, and let $C_h$ be the Clifford algebra defined by $h$ over $\mathfrak{G}'$; here $C_h$ is the complex Clifford algebra whose real covectors are self-adjoint and satisfy $j_h(v)^2=h(v,v)1$.

For our purpose, we wish to have a systematic way to choose a representation of $C_h$ regardless of the choice of $h$, which is possible since the isomorphism class of $C_h$ only depends on $\dim\mathfrak{G}$.

If $\dim\mathfrak{G}$ is even, then $C_h$ is isomorphic to the full matrix algebra of $2^{\dim\mathfrak{G}/2}\times 2^{\dim\mathfrak{G}/2}$ matrices; let $\pi$ be its unique irreducible representation up to unitary conjugation, and fix some nonzero natural number $p$.

If $\dim\mathfrak{G}$ is odd, then $C_h$ is the direct sum $M\oplus M$ of two full matrix algebras of matrices of size $2^{(\dim\mathfrak{G}-1)/2}\times 2^{(\dim\mathfrak{G}-1)/2}$; let $\pi_1$ and $\pi_2$ be the unique irreducible representations, up to unitary conjugation, of the first and second summand respectively, and fix some nonzero natural numbers $p$ and $q$.

We now let $c_h$ be (up to unitary equivalence) the finite dimensional *-representation of $C_h$ on some Hermitian space $S_h$ obtained as the direct sum of $p$ times the representation $\pi$ (if $\dim{\mathfrak{G}}$ is even) or $p$ times the representation $\pi_1$ plus $q$ times the representation $\pi_2$. Note that $c_h$ is faithful by construction.

We define an inner product $\inner{\cdot}{\cdot}{W_h}$ on $\A_\infty\otimes S_h$ by setting, for all $a\otimes s, b\otimes t \in \A_\infty\otimes S_h$,
\begin{equation*}
	\inner{a\otimes s}{b\otimes t}{W_h} \coloneqq \inner{s}{t}{S_h}\eta(b^\ast a) \text.
\end{equation*}
Since $\eta$ is faithful, $\A_\infty\otimes S_h$ canonically injects in its Hilbert
space completion $W_h=L^2(\A,\eta)\otimes S_h$.

For $a\in\A_\infty$, write $da\in\A_\infty\otimes\mathfrak{G}'$ for the linear map
$X\mapsto d_Xa$. Let $j_h:\mathfrak{G}'\rightarrow C_h$ be the canonical
inclusion and let $m_h(c\otimes s)=c_h(c)s$. We define the operator $D_h$ on the dense subspace  $\A_\infty\otimes S_h$ of the Hilbert space $W_h$ as the composition
\begin{equation*}
	\A_\infty\otimes S_h
	\xrightarrow{d\otimes\mathrm{id}}
	\A_\infty\otimes\mathfrak{G}'\otimes S_h
	\xrightarrow{\mathrm{id}\otimes j_h\otimes\mathrm{id}}
	\A_\infty\otimes C_h\otimes S_h
	\xrightarrow{-i\,\mathrm{id}\otimes m_h}
	\A_\infty\otimes S_h.
\end{equation*}
If $e_1(h),\ldots,e_d(h)$ is an oriented orthonormal basis of
$\mathfrak{G}$ for the inner product $h^\vee$ dual to $h$, then
$e^1(h),\ldots,e^d(h)$ is its dual basis and is orthonormal for $h$.
On smooth vectors, the composition gives
\begin{equation*}
	D_h=-i\sum_{j=1}^d d_{e_j(h)}\otimes c_h(e^j(h)).
\end{equation*}
The factor $-i$ makes $D_h$ symmetric for our Clifford convention. In fact, $D_h$ is essentially self-adjoint \cite[Proposition~4.1]{Gabriel16}. Its self-adjoint closure
$\Dirac_h$ has compact resolvent by \cite[Theorem~5.5(i)]{Gabriel16}, applied to the ergodic action of $G^\circ$. The resulting spectral triple $(\A,W_h,\Dirac_h)$ is metric by \cite[Theorem~4.2]{Rieffel98a} applied to the ergodic action of the connected compact Lie group $G^\circ$. We record below the graph estimates, including the effect of noncommuting Lie algebra generators, needed for our continuity argument.
\medskip

	We now transport the operators to a common Hilbert space using our choice of Clifford representations. Fix an inner product $g$ on $\mathfrak{G}'$ and an oriented $g^\vee$-orthonormal basis 	$e_1(g),\ldots,e_d(g)$ of $\mathfrak{G}$. For any inner product $h$ on
	$\mathfrak{G}'$, let
\begin{equation*}
 G_{h,j}\coloneqq
 \begin{pmatrix}
  h^\vee(e_1(g),e_1(g)) & \cdots & h^\vee(e_1(g),e_j(g))\\
  \vdots                & \ddots & \vdots\\
  h^\vee(e_j(g),e_1(g)) & \cdots & h^\vee(e_j(g),e_j(g))
 \end{pmatrix}
\end{equation*}
be the Gram matrix of $e_1(g),\ldots,e_j(g)$ for $h^\vee$.
	
	 Let $M_{h,j}$ be obtained from $G_{h,j}$ by replacing its last row with $(1,\ldots,1)$, and denote by $M_{h,j}^{j,k}$ the matrix obtained by removing row $j$ and column $k$. With the convention that empty determinants equal $1$, the Gram--Schmidt formula is
	\begin{equation*}
		e_j(h)=\frac{\sum_{k=1}^j(-1)^{j+k}\det(M_{h,j}^{j,k})e_k(g)}
		{\sqrt{\det(G_{h,j})\det(G_{h,j-1})}}.
	\end{equation*}
	Thus $e_j(h)=\sum_{k=1}^j f_{j,k}(h)e_k(g)$, where
	\begin{equation}\label{gram-f-eq}
		f_{j,k}(h)=\frac{(-1)^{j+k}\det(M_{h,j}^{j,k})}
		{\sqrt{\det(G_{h,j})\det(G_{h,j-1})}}.
	\end{equation}
	All these functions are continuous, since the Gram determinants are strictly positive. Moreover, $f_{j,k}(g)=\delta_{jk}$.
	
	\begin{remark}
	To verify the formula, let $v_j$ denote its numerator. By cofactor
expansion along the last row, $h^\vee(v_j,e_r(g))$ is the determinant
of the matrix obtained from $G_{h,j}$ by replacing its last row with
its $r$-th row. Thus
\begin{equation*}
 h^\vee(v_j,e_r(g))=
 \begin{cases}
  0 & \text{if } 1\leq r<j,\\
  \det G_{h,j} & \text{if } r=j.
 \end{cases}
\end{equation*}
The coefficient of $e_j(g)$ in $v_j$ is $\det G_{h,j-1}$. Hence
\begin{equation*}
 h^\vee(v_j,v_j)
 =\det G_{h,j-1}\,h^\vee(v_j,e_j(g))
 =\det G_{h,j-1}\det G_{h,j}.
\end{equation*}
Since $\det G_{h,j-1}>0$, dividing $v_j$ by its norm gives the $j$-th Gram--Schmidt vector with positive coefficient along $e_j(g)$.
	\end{remark}
	
The isometry of dual spaces sending $e^j(h)$ to $e^j(g)$ induces
a *-isomorphism $\beta_g^h:C_h\rightarrow C_g$. In odd dimension,
we label the two matrix summands at each metric by the signs of
the normalized Clifford volume element for the ordered dual frame.
Since $\beta_g^h$ sends each $e^j(h)$ to $e^j(g)$, it sends the
volume element for $h$ to the volume element for $g$. Thus $c_h$
and $c_g\circ\beta_g^h$ have the same irreducible multiplicities,
so there is a unitary $V_g^h:S_h\rightarrow S_g$ such that
\begin{equation*}
 V_g^h c_h(\cdot)(V_g^h)^*=c_g(\beta_g^h(\cdot)).
\end{equation*}

	Let $U_g^h=1\otimes V_g^h:W_h\rightarrow W_g$. Then
	\begin{equation*}
		\Dirac_g^h=U_g^h\Dirac_h(U_g^h)^*
		=-i\sum_{j=1}^d d_{e_j(h)}\otimes C_j \text{, where }
		\quad C_j=c_g(e^j(g)).
	\end{equation*}
	The unitary $U_g^h$ intertwines the representations of $\A$, so the spectral triples $(\A,W_h,\Dirac_h)$ and $(\A,W_g,\Dirac_g^h)$ are unitarily equivalent. The expression for $D_h$ is independent of the choice of orthonormal basis. The transported expression for $\Dirac_g^h$ uses the frames chosen to define $\beta_g^h$.

	\medskip

	For each $j\in\{1,\ldots,d\}$, let $A_j$ be the skew-adjoint generator on $W_g$ of the one-parameter unitary group induced by $t\in\R\mapsto\exp(t e_j(g))$. On $\A_\infty\otimes S_g$, this generator acts as
\begin{equation*}
		A_j\left|_{\A_\infty\otimes S_g}\right. = d_{e_j(g)}\otimes1\text.
	\end{equation*}

	We also write
	\begin{equation*}
		[e_j(g),e_k(g)]=\sum_{r=1}^d b_{jk}^r e_r(g)
	\end{equation*}
	 for all $j\in\{1,\ldots,d\}$. On $\A_\infty\otimes S_g$,
	\begin{equation*}
		\Dirac_g^2=-\sum_j A_j^2-\sum_{j<k,r}b_{jk}^r A_r(1\otimes C_jC_k).
	\end{equation*}

	Set, for all $\xi \in \A_\infty\otimes S_g$,  
	\begin{equation*}
		S(\xi)=(\sum_j\norm{A_j\xi}{}^2)^{1/2}\qquad\text{ and }\qquad B_g=\sum_{\substack{1\leq j<k \leq d \\1\leq r \leq d}}|b_{jk}^r|\text.
	\end{equation*}
	Since each $C_j$ is unitary, we obtain
	\begin{equation*}
		S(\xi)^2\leq\norm{\Dirac_g\xi}{}^2+B_gS(\xi)\norm{\xi}{},
		\qquad
		S(\xi)\leq\norm{\Dirac_g\xi}{}+B_g\norm{\xi}{}.
	\end{equation*}
Conversely, $\norm{D_g\xi}{}\leq\sqrt d\,S(\xi)$. Thus the graph norm of $D_g$ on $\A_\infty\otimes S_g$ is equivalent to $\norm{\cdot}{}+S(\cdot)$. Write  $D_g^h=U_g^h D_h (U_g^h)^*$ on $\A_\infty\otimes S_g$. The same calculation in the $h$-frame shows that the graph norm of $D_g^h$ is equivalent to
\begin{equation*}
\xi\in\A_\infty\otimes S_g \mapsto \norm{\xi}{}+\left(\sum_{j=1}^d\norm{(d_{e_j(h)}\otimes1)\xi}{}^2\right)^{1/2}.
\end{equation*}
On $\A_\infty\otimes S_g$, the invertible change of basis from
$e_1(g),\ldots,e_d(g)$ to $e_1(h),\ldots,e_d(h)$ shows that this
norm is equivalent to $\norm{\xi}{}+S(\xi)$.

We now identify the completion for this norm. Set
\begin{equation*}
 E=\bigcap_{j=1}^d\dom{A_j},
\end{equation*}

Since the $A_j$ are closed, $E$ is complete for the norm $\norm{\cdot}{}+S(\cdot)$. Let $(P_N)_{N\in\N}$ be the projections onto increasing finite sums of Peter--Weyl isotypic components for the action of $G^\circ$ on $W_g$, exhausting the
decomposition (so $P_N$ strongly converge to the identity). These components are finite dimensional by ergodicity \cite{Wassermann88} and consist of vectors in $\A_\infty\otimes S_g$. The projections commute with the unitary actions generated by $A_1, \ldots, A_d$,  and hence with its infinitesimal generators, on their domains. Thus, for $\xi\in E$ and $j\in\{1,\ldots,d\}$,
\begin{equation*}
 \lim_{N\to\infty} P_N\xi = \xi,\qquad
 \lim_{N\to\infty} A_j P_N\xi= (\lim_{N\to\infty}P_N) A_j\xi = A_j\xi.
\end{equation*}

Consequently, $\A_\infty\otimes S_g$ is dense in $E$ for $\norm{\cdot}{}+S$. By the graph-norm equivalences above, the closures of $D_g$ and $D_g^h$ therefore have domain $E$. In
particular, $\dom{\Dirac_g}=\dom{\Dirac_g^h}=E$.
	
		Let $T=\Dirac_g^h-\Dirac_g$ on the common domain $E$. Note that $T$ is symmetric by construction. Now, each Peter--Weyl isotypic component of $W_g$ for the action of $G^\circ$ is contained in $E$ and invariant under both $\Dirac_g$ and $\Dirac_g^h$. Since $T$ is symmetric, its restriction to each of these finite-dimensional components is self-adjoint. Thus, $T\pm i$ maps each component onto itself. It follows that $\operatorname{ran}(T\pm i)$ contains the algebraic sum of the isotypic components, which is dense in $W_g$. Hence $T$ is
essentially self-adjoint on $E$.
	
Using Eq.~\eqref{gram-f-eq}, set
\begin{equation*}
 f(h)\coloneqq\sum_{j=1}^d\sum_{k=1}^j
 |f_{j,k}(h)-\delta_{jk}|,
 \qquad K_g\coloneqq\max\{1,B_g\}.
\end{equation*}
On $\A_\infty\otimes S_g$, we have
\begin{equation*}
 D_g^h-D_g
 =-i\sum_{j=1}^d\sum_{k=1}^j
 (f_{j,k}(h)-\delta_{jk})(1\otimes C_j)A_k.
\end{equation*}
Since each $C_j$ is unitary and $\norm{A_k\xi}{}\leq S(\xi)$,
\begin{align*}
 \norm{(D_g^h-D_g)\xi}{}
 &\leq\sum_{j=1}^d\sum_{k=1}^j
 |f_{j,k}(h)-\delta_{jk}|\norm{A_k\xi}{}\\
 &\leq f(h)S(\xi)\\
 &\leq f(h)\bigl(\norm{D_g\xi}{}+B_g\norm{\xi}{}\bigr)\\
 &\leq K_gf(h)\CDN_g(\xi).
\end{align*}
By density of $\A_\infty\otimes S_g$ in $E$ for the common
graph norm, this gives, for every $\xi\in E$,
\begin{equation*}
 \norm{(\Dirac_g^h-\Dirac_g)\xi}{}
 \leq K_gf(h)\CDN_g(\xi).
\end{equation*}
Finally, continuity of the coefficients $f_{j,k}$ and
$f_{j,k}(g)=\delta_{jk}$ give $\lim_{h\rightarrow g}f(h)=0$. Thus, Condition (2) of Lemma \ref{main-lemma} holds.

	\medskip

	We next compare the commutator seminorms. The finite-sum formulas for $D_g$ and $D_g^h$ extend to $\Dirac_g$ and $\Dirac_g^h$ on $E$, since $\A_\infty\otimes S_g$ is dense in $E$ for the joint graph norm and each $A_k$ is continuous from $E$ with this norm to $W_g$.
	
	 The Clifford operators preserve $E$ and commute with the representation of $\A$. For every $a$ preserving $E$, the identity
	\begin{equation*}
		(1\otimes C_k)[\Dirac_g,a]+[\Dirac_g,a](1\otimes C_k)=-2i[A_k,a]
	\end{equation*}
	shows that boundedness of $[\Dirac_g,a]$ implies boundedness of every $[A_k,a]$, with
	$\norm{[A_k,a]}{}\leq\Lip_g(a)$.
	The identical argument in the $h$-basis and the invertible change of basis show that the domains of $\Lip_h$ and $\Lip_g$ agree. On this domain,
	\begin{align*}
		\norm{[\Dirac_g^h-\Dirac_g,a]}{}
		 & \leq\sum_{j=1}^d\sum_{k=1}^j
		|f_{j,k}(h)-\delta_{jk}|\norm{[A_k,a]}{} \\
		 & \leq f(h)\Lip_g(a).
	\end{align*}
	Consequently $|\Lip_h(a)-\Lip_g(a)|\leq f(h)\Lip_g(a)$, and Condition (1) of Lemma (\ref{main-lemma}) holds as well.

	\begin{theorem}\label{lie-cont-thm}
		Let $\alpha$ be a strongly continuous action of a compact Lie group $G$ on a unital C*-algebra $\A$ whose restriction to the identity component is ergodic. With the invariant state and Clifford representation multiplicities fixed as above, the map
		\begin{equation*}
			h\in\mathscr{I}\longmapsto(\A,W_h,\Dirac_h)
		\end{equation*}
		from the space $\mathscr{I}$ of inner products on $\mathfrak{G}'$ to the space of metric spectral triples endowed with the spectral propinquity is continuous.
	\end{theorem}
	\begin{proof}
		The estimates just established verify both hypotheses of Lemma (\ref{main-lemma}) on $W_g$, for every $g\in\mathscr{I}$. Unitary equivalence then gives the stated continuity.
	\end{proof}

	\begin{corollary}
		If $h_n\rightarrow g$ in $\mathscr{I}$, then, in every compact interval with endpoints outside $\spectrum{\Dirac_g}$, the eigenvalues of $\Dirac_{h_n}$ converge to those of $\Dirac_g$, counted with multiplicity, in the sense of Theorem (\ref{cont-thm}).
	\end{corollary}
	\begin{proof}
		Apply Theorem (\ref{cont-thm}) to Theorem (\ref{lie-cont-thm}).
	\end{proof}

	\begin{remark}
		The topology on $\mathscr{I}$ is the topology obtained by identifying inner products with positive definite real matrices and using any matrix norm. For connected $G$, our hypothesis is simply ergodicity. For disconnected $G$, ergodicity of the identity component is necessary for metricity: the infinitesimal generators vanish on elements fixed by the identity component. The examples of this section include quantum tori; quantum solenoids require actions of compact groups which need not be Lie groups and are treated by different constructions, e.g. \cite{Latremoliere24a}.
	\end{remark}

\section{Continuity from the $C^1$ topology on Riemannian metrics for the spectral propinquity}

	We now investigate the continuity, for the spectral propinquity, of almost commutative models over a fixed spin manifold, as a function of the Riemannian metric and the Dirac operator on the finite dimensional factor.

	We begin with the commutative factor. We prove that the canonical metric spectral triple varies continuously when the metric varies in the $C^1$ topology. The continuity of eigenvalues is then a consequence of Theorem (\ref{cont-thm}).

	We will rely upon computations made in \cite{BG} when convenient. We give the explicit formulas needed to apply Lemma (\ref{main-lemma}) and to handle products with finite dimensional spectral triples. For the translation action on $\A=C(G)$ with $G$ connected, the cotangent bundle is trivialized by translations, and the spinor module used in the previous section is free. A single inner product on $\mathfrak{G}'$ 	determines the metric data in every fiber in this trivialization. On a 	general spin manifold, the spinor bundle may be nontrivial, and the	metric data vary from fiber to fiber. Their first derivatives enter the spin connection, leading to the $C^1$ topology.

	For the translation action, the fixed invariant state is integration against normalized Haar measure. On the spin manifold the volume form varies with the Riemannian metric. The Gram--Schmidt construction in the previous section uses a fixed basis of $\mathfrak{G}$. Here we follow \cite{BG} and use the positive square root relating the fiberwise inner products to construct a bundle identification independent of local frames. The associated unitary on square integrable spinors also accounts for the change in volume form.

\medskip

Let us henceforth fix a connected, compact, spin closed manifold $M$ of dimension $d > 0$ (the case of dimension $0$ being trivial here). Let $g$ be a Riemannian metric over $M$. We denote the spinor bundle over $M$ associated with $g$ by $\Sigma^g M$. Now, denote by $\Gamma_{\infty}(\Sigma^g M)$ the $C^\infty(M)$-module of smooth (i.e. $C^\infty$) sections of $\Sigma^g M$. Our base reference for these constructions is \cite{Friedrich}.

Let $\Cl{g}$ be the vector bundle over $M$ whose fiber at $m \in M$ is the Clifford algebra $\Cl{T_m M,g_m}$ of the fiber $T_m M$ with the inner product $g_m$. For any smooth section $\gamma$ of $\Cl{g}$ and smooth spinor field $s$, we define the pointwise action of $\gamma$ on $s$ by
\begin{equation*}
	c(\gamma)s : m \in M \mapsto \gamma_m s(m) \in \Sigma^g_m M \text;
\end{equation*}
of course, $c(\gamma)$ maps $\Gamma_\infty(\Sigma^g M)$ to itself. Let us also define the map $m : \Gamma_\infty(\Cl{g}) \otimes_{C^\infty(M)} \Gamma_\infty(\Sigma^g M) \rightarrow \Gamma_\infty(\Sigma^g M)$ by setting, for all $\gamma \in \Gamma_\infty(\Cl{g})$ and $s \in \Gamma_\infty(\Sigma^g M)$
\begin{equation*}
	m(\gamma,s) \coloneqq c(\gamma)s  \text.
\end{equation*}
We define the \emph{Atiyah--Singer} Dirac operator $D_g$ as the composition:
	\begin{equation}\label{Dirac-def}
		\begin{aligned}
			D_g:\Gamma_\infty(\Sigma^g M)
			 & \xrightarrow{\slashed\nabla^g}
			\Gamma_\infty(T^*M)\otimes_{C^\infty(M)}\Gamma_\infty(\Sigma^g M) \\
			 & \xrightarrow{j\otimes\mathrm{id}}
			\Gamma_\infty(\Cl g)\otimes_{C^\infty(M)}\Gamma_\infty(\Sigma^g M)
			\xrightarrow{m}\Gamma_\infty(\Sigma^g M),
		\end{aligned}
	\end{equation}

where $\slashed{\nabla}^g$ is the spin connection (i.e. the lift of the Levi-Civita connection for $g$ to$\Sigma^g M$), and $j$ is the composition of the musical morphism sending $\omega \in T^\ast M$ to $X \in TM$ such that $g(X,\cdot) = \omega$, with the canonical injection of $TM$ into $\Cl{g}$. We will work with the closure of $D_g$, denoted by $\Dirac_g$, which is a self-adjoint operator on a dense domain $\Gamma_{H_1}(\Sigma^g M)$ of the Hilbert space $\Gamma_{L^2}(\Sigma^g M)$ obtained by completing $\Gamma_\infty(\Sigma^g M)$ for its natural inner product (described below). Of course, $\Gamma_{H_1}(\Sigma^g M)$ is a Banach space for the graph norm of $\Dirac_g$ since $\Dirac_g$ is self-adjoint.

The dependency of the Dirac operator on the underlying Riemannian metric is the subject of many papers, including \cite{BG,Nowaczyk13,Maier97}. Of core interest for us is the following result \cite[Theorem 2.1]{Nowaczyk13}, which in turn uses \cite{BG,Maier97}. In this work, the space $\mathcal{R}(M)$ of Riemannian metrics over $M$ is endowed with the so-called $C^1$-topology, which is informally the topology where, in each chart of $M$, the components of the metrics vary according to the $C^1$ norm. We will give a detailed presentation of this topology later in this section.

\begin{theorem}[\cite{Nowaczyk13}]\label{BG-thm}
	Fix a Riemannian metric $g \in \mathcal{R}(M)$. For every $h \in \mathcal{R}(M)$, there exists a unitary $\widehat{\beta_g^h} : \Gamma_{L^2}(\Sigma^h M) \rightarrow \Gamma_{L^2}(\Sigma^g M)$ such that, if we set:
	\begin{equation*}
		\Dirac_g^h \coloneqq \widehat{\beta_g^h} \Dirac_h (\widehat{\beta_g^h})^\ast
	\end{equation*}
	then $\dom{\Dirac_g^h} = \dom{\Dirac_g}$, and the map
	\begin{equation*}
		h \in \mathcal{R}(M) \longmapsto \Dirac_g^h
	\end{equation*}
	is continuous from $\mathcal{R}(M)$ to $\B(\Gamma_{H^1}(\Sigma^g M), \Gamma_{L^2}(\Sigma^g M))$.
\end{theorem}

The graph norm of any Dirac operator is equivalent to a Sobolev norm \cite{Maier97}, thus we henceforth use the graph norm $\CDN_g$ of the Dirac operator $\Dirac_g$ to metrize $\Gamma_{H_1}(\Sigma^g M)$, as mentioned above. In other words, \cite[Theorem 2.1]{Nowaczyk13} proves that the function
\begin{equation*}
	h \in \mathcal{R}(M) \mapsto \sup_{\substack{ \xi \in \dom{\Dirac_g} \\ \CDN_g(\xi) \leq 1 }}\norm{ (\Dirac_g^h - \Dirac_g) \xi}{\Hilbert_g}
\end{equation*}
is continuous.

\medskip

	This result gives the graph norm estimate required by Lemma (\ref{main-lemma}) (and motivated it initially). We derive the estimate directly from the local	expression of the conjugated Dirac operator. This computation is independent of the spectral perturbation results of \cite{BG} and will also be used for almost commutative spectral triples. We then verify the corresponding estimate for the commutator seminorms.

\medskip

It will prove helpful to consider the local expression of the Dirac operator, so that its dependency on the metric is made apparent. To this end, let us fix an explicit finite atlas $\mathcal{A}$ of $M$, with the property that, for any chart $(U,\varphi)$ in $\mathcal{A}$,  there exists a contractible coordinate neighborhood $(V,\psi)$ of $M$ such that $\closure{U}\subseteq V$ and $\varphi = \psi_{|U}$ --- such an atlas exists since $M$ is compact. When choosing a local frame over $U$ for some bundle over $M$, we always will mean the restriction to $U$ of some frame over $V$.  As a matter of notation, we will write a chart in $\mathcal{A}$ as $(U,x^1,\ldots,x^d)$ with $x^1,\ldots,x^d : U \rightarrow \R$, and we denote the associated canonical, coordinate frame of $TM_{|U}$ as $\left(\frac{\partial}{\partial x^1}, \ldots, \frac{\partial}{\partial x^d}\right)$. The dual basis, or coframe, for the cotangent subbundle $T^\ast U$ of $TM$ is denoted by $\dd{x^1},\ldots,\dd{x^d}$.

We choose $\mathcal{A}$ to be an oriented atlas, with each larger
chart $(V,\psi)$ also compatible with the orientation of $M$. Over each contractible $V$, we choose a spin lift of the oriented $g$-orthonormal frame $e_1(g),\ldots,e_d(g)$ used below to compute the spin connection.

We recall from \cite{Friedrich} the construction of the spin connection on $\Sigma^g M$. We proceed locally. Let $(U,x^1,\ldots,x^d)$ be a chart of $M$ in our atlas $\mathcal{A}$. We write
\begin{equation*}
	g_{|U} = \sum_{1\leq j,k \leq d} g_{jk} \dd{x^j}\otimes \dd{x^k}\text,
\end{equation*}
i.e. we identify the metric $g_{|U}$ with a section of positive definite matrices over each fiber in $TM_{|U}$.

Let $e_1(g),\ldots,e_d(g)$ be an orthonormal frame of $TM_{|U}$ for the metric $g$ (again by convention, this frame extends to some open subset $V$ of $M$ with $\closure{U}\subseteq V$). Such a frame can be constructed by applying the Gram-Schmidt process to the local frame $\frac{\partial}{\partial x^1},\ldots,\frac{\partial}{\partial x^d}$, though any orthonormal local frame will do. We write $e_j(g) = \sum_{k} e_j^k(g) \frac{\partial}{\partial x^k}$, with $e_j^k$ smooth functions over $U$ for all $j,k\in\{1,\ldots,d\}$. \emph{For each chart in $\mathcal{A}$, we henceforth fix such an orthonormal frame all throughout this section.}

Let $\nabla^g$ be the Levi-Civita connection over $M$ for $g$. Locally, we have
\begin{equation*}
	\nabla^g_X(Y) = \sum_{k=1}^d\left(\sum_{j=1}^d X^j \frac{\partial Y^k}{\partial x^j} + \sum_{j,r=1}^d\Gamma(g)_{jr}^{k} X^j Y^r\right)\frac{\partial}{\partial x^k} \text,
\end{equation*}
for all $X = \sum_{j=1}^d X^j \frac{\partial}{\partial x^j} \in TM_{|U}$ (i.e. $X^j = \dd{x^j}(X)$) and $Y = \sum_{j=1}^d Y^j \frac{\partial}{\partial x^j} \in TM_{|U}$, where the scalar coefficients $\Gamma$'s are the Christoffel symbol given by:
\begin{equation}\label{Christoffel-eq}
	\Gamma(g)_{jq}^p \coloneqq \frac{1}{2}\sum_{l=1}^d (g^{-1})_{p l} \left( \frac{\partial g_{l q}}{\partial x_j} + \frac{\partial g_{jl}}{\partial x_q} - \frac{\partial g_{jq}}{\partial x_l} \right) \text,
\end{equation}
where $g^{-1}$ is meant here for $m \in U \mapsto (g_{jk})_{1\leq j,k\leq d}^{-1}$.

\medskip

	For all $j,k,l\in\{1,\ldots,d\}$, set
	\begin{equation*}
		\omega_{jkl}(g)=g(\nabla^g_{e_j(g)}e_k(g),e_l(g)).
	\end{equation*}
	A direct computation in local coordinates gives
	\begin{equation*}
		\nabla^g_{e_j}e_k=\sum_{q,m=1}^d e_j^m
		\left(\frac{\partial e_k^q}{\partial x^m}
		+\sum_{p=1}^d e_k^p\Gamma(g)_{mp}^q\right)\frac{\partial}{\partial x^q},
	\end{equation*}
	and therefore
	\begin{equation}\label{spin-eq}
		\omega_{jkl}(g)=\sum_{q,r,m=1}^d g_{qr}e_l^r e_j^m
		\left(\frac{\partial e_k^q}{\partial x^m}
		+\sum_{p=1}^d e_k^p\Gamma(g)_{mp}^q\right).
	\end{equation}
	We use the geometric Clifford convention
	$c(e_j)c(e_k)+c(e_k)c(e_j)=-2\delta_{jk}$ in this section.
	The spin connection on $\Sigma^g M$ is given in the chosen spin frame by
	\begin{equation*}
		\slashed{\nabla}_{e_j}^g\psi
		=\sum_{p=1}^d e_j^p\frac{\partial\psi}{\partial x^p}
		+\frac14\sum_{k,l=1}^d\omega_{jkl}(g)c(e_k)c(e_l)\psi,
	\end{equation*}
and, of course, $\slashed{\nabla}_{X}(\psi) = \sum_{j=1}^d X_e^j \slashed{\nabla}_{e_j}(\psi)$ whenever $X \in TM$, supported on $U$, is written as $X = \sum_{j=1}^d X_e^j e_j$ (note that here, we express $X$ over the frame $(e_1,\ldots,e_d)$ rather than the coordinate frame $\left(\frac{\partial}{\partial x^1},\ldots,\frac{\partial}{\partial x^d}\right)$).

We refer to \cite{Friedrich} for the claim that $\slashed{\nabla}^g$ thus defined is indeed a connection on $\Gamma_\infty(\Sigma^g M)$.
\medskip

We may now express the Atiyah--Singer Dirac operator $D_g$ on the smooth sections of $\Sigma^g M$ supported on $U$ locally as:
\begin{align}\label{Dg-eq}
	D_g
	 & \coloneqq \sum_{j=1}^d c(e_j(g)) \slashed{\nabla}^g_{e_j(g)}                                                                                                      \\
	 & =\sum_{j=1}^d c(e_j(g)) \left[\sum_{p=1}^d e_j^p(g) \frac{\partial}{\partial x^p} + \frac{1}{4}\sum_{k,l=1}^d \omega_{jkl}(g) c(e_k(g) e_l(g))  \right] \nonumber
\end{align}
which follows from the composition given in Expression \eqref{Dirac-def}. For our purpose, we wish to look at $D_g$ as an unbounded operator acting on a Hilbert space, which of course is  completion of $\Gamma_\infty(\Sigma^g M)$ for an inner product which we now describe.

\medskip

There is an inner product $\cdot_m$ on each fiber $\Sigma_m^g M$ of the spinor bundle $\Sigma^g M$  by construction. Now, let
\begin{equation*}
	\mathrm{vol}_g = \sqrt{\det(g_{kl})_{1\leq k,l\leq d}} \; \dd{x^1}\wedge\cdots\wedge \dd{x^d} \end{equation*}
be the volume $d$-form induced by $g$ over $M$. We then define an inner product on $\Gamma_\infty(\Sigma^g M)$ by
\begin{equation*}
	\inner{s}{s'}{L^2(g)} \coloneqq \int_M s_m \cdot_m s'_m  \dd{\mathrm{vol}_g(m)}.
\end{equation*}
The completion of $\Gamma_\infty(\Sigma^g M)$ for $\inner{\cdot}{\cdot}{L^2(g)}$ is a Hilbert space denoted by $\Gamma_{L^2}(\Sigma^g M)$ (note that indeed this inner product is definite positive on $\Gamma_\infty(\Sigma^g M)$). The operator $D_g$ is essentially self-adjoint on this space, and we denote its self-adjoint closure by $\Dirac_g$. We also denote its domain by $\Gamma_{H^1}(\Sigma^g M)$, which is a Banach space when endowed with the graph norm of $\Dirac_g$, defined by $\CDN_g(\xi) \coloneqq \norm{\xi}{\Gamma_{L^2}(\Sigma^g M)} + \norm{\Dirac_g\xi}{\Gamma_{L^2}(\Sigma^g M)}$ for all $\xi \in \Gamma_{H_1}(\Sigma^g M)$.

\medskip

For the computations that follow, we observe that we can bound the Hilbert norm on $\Gamma_{L^2}(\Sigma^g M)$ locally, by noting that
\begin{equation*}
	\norm{\psi}{\Gamma_{L^2}(\Sigma^g M)} \leq \sum_{(U,x^1,\ldots,x^d)\in\mathcal{A}} \norm{\psi_{|U}}{\Gamma_{L^2}(\Sigma^g M)_{|U}}
\end{equation*}
where we used the fact that $\mathcal{A}$ is finite. So as a rule for the rest of this section, it will be sufficient to find local bounds over each chart for this norm to conclude that the sum of these local bounds is a global bound on the Hilbert norm (or Sobolev norm).

\medskip

	For a smooth spinor $\psi$, define, using the fixed local spin frames,
	\begin{equation}\label{H1-norm-eq}
		\norm{\psi}{H_1}=\norm{\psi}{L^2(g)}+
		\sum_{U\in\mathcal A}\sum_{p=1}^d
		\norm{\frac{\partial\psi}{\partial x^p}}{L^2(U,g)}.
	\end{equation}
	Here and below a sum over $U\in\mathcal A$ means a sum over the charts of our fixed atlas. The coefficients in \eqref{Dg-eq} are bounded on every $\closure{U}$. Taking local $L^2$ norms and then summing gives a constant
	$C(g)>0$ such that
	\begin{equation}\label{one-bound-Dg-H1-eq}
		\CDN_g(\psi)\leq C(g)\norm{\psi}{H_1}.
	\end{equation}

	For the reverse inequality we use the Schr\"odinger--Lichnerowicz identity \cite{Friedrich}. Since $M$ is closed, integration gives
	\begin{equation*}
		\norm{\slashed\nabla^g\psi}{L^2(g)}^2
		=\norm{D_g\psi}{L^2(g)}^2
		-\frac14\int_M\operatorname{scal}_g|\psi|_g^2\,\mathrm{vol}_g
		\leq\norm{D_g\psi}{L^2(g)}^2+\kappa_g\norm{\psi}{L^2(g)}^2,
	\end{equation*}
	where $\kappa_g=\frac14\norm{\operatorname{scal}_g}{C(M)}$.
	In each fixed local frame,
	$\slashed\nabla^g_{\partial/\partial x^p}=\partial/\partial x^p+A_{U,p}$,
	where $A_{U,p}$ is a smooth matrix function bounded on $\closure{U}$. Hence
	\begin{equation*}
		\norm{\frac{\partial\psi}{\partial x^p}}{L^2(U,g)}
		\leq \sup_U\left|\frac{\partial}{\partial x^p}\right|_g
		\norm{\slashed\nabla^g\psi}{L^2(g)}
		+\norm{A_{U,p}}{\infty}\norm{\psi}{L^2(g)}.
	\end{equation*}
	Summing these inequalities proves, for some $Q(g)>0$,
	\begin{equation}\label{equiv-H1-CDN-eq}
		Q(g)\norm{\psi}{H_1}\leq\CDN_g(\psi)
		\leq C(g)\norm{\psi}{H_1}.
	\end{equation}
	The completion of smooth spinors for $\norm{\cdot}{H_1}$ is the usual first Sobolev space: a smooth partition of unity subordinate to the finite atlas and local mollification give density in the space of spinors with square integrable weak first derivatives. The displayed equivalence therefore identifies this completion with $\dom{\Dirac_g}$ and extends both inequalities to that domain.

	\medskip

Now, our purpose is to study the dependency of the Dirac operator $\Dirac_g$ in the metric $g$. To this end, let us work on the space $\mathcal{R}(M)$ of all Riemannian metrics over our manifold $M$. We endow $\mathcal{R}(M)$ with the so-called $C^1$-topology, which is metrized using the restriction of a norm on the space of $(0,2)$-tensors to the subspace $\mathcal{R}(M)$. It will be helpful for our exposition to describe this metric here.

Fix $g, h \in \mathcal{R}(M)$. Let $(U,(x_1,\ldots,x_d))$ be a local chart of $M$ in $\mathcal{A}$, and let $\varepsilon > 0$. In these coordinates, we write
\begin{equation*}
	g_{|U} = \sum_{j,k=1}^d g_{jk} dx^j \otimes dx^k \text{ and } h_{|U} = \sum_{j,k=1}^d h_{jk} dx^j \otimes dx^k \text,
\end{equation*}
with $g_{jk} : U \rightarrow\R$, $h_{jk} : U\rightarrow \R$ are $C^\infty$ functions over $U$. We then define:
\begin{equation*}
	\delta_{U,x_1,\ldots,x_d} (g,h) \coloneqq \max_{j,k \in \{1,\ldots,d\}} \norm{ h_{jk} - g_{jk}}{C_b^1(U)}
\end{equation*}
and
\begin{equation*}
	\rho_{U,x_1,\ldots,x_d}(g,h)  \coloneqq \max_{j,k \in \{1,\ldots,d\}} \norm{ h_{jk} - g_{jk}}{C_b(U)}
\end{equation*}
where
\begin{equation*}
	\norm{f}{C_b(U)} \coloneqq \sup_{x\in U} |f(x)| \text,
\end{equation*}
and
\begin{equation*}
	\norm{f}{C_b^1(U)} \coloneqq \norm{f}{C_b(U)} + \sup_{j \in \{1,\ldots,d\}} \norm{\frac{\partial f}{\partial x_j}}{C_b(U)} \text,
\end{equation*}
for $C^1$ functions with bounded value and first derivatives. In particular, these norms are finite for the metric coefficients used here, since they extend smoothly to a neighborhood of $\closure{U}$.

\medskip

The \emph{$C^1$}-topology on $\mathcal{R}(M)$ is the topology induced by the metric:
\begin{equation*}
	\delta_{\mathcal{A}}: h,g \in \mathcal{R}(M) \mapsto \max_{ (U,x_1,\ldots,x_d) \in \mathcal{A}} \delta_{U,x_1,\ldots,x_d} (g,h) \text.
\end{equation*}
The $C^0$-topology on $\mathcal{R}(M)$ is the topology induced by the metric:
\begin{equation*}
	\rho_{\mathcal{A}}: h,g \in \mathcal{R}(M) \mapsto \max_{ (U,x_1,\ldots,x_d) \in \mathcal{A}} \rho_{U,x_1,\ldots,x_d} (g,h) \text.
\end{equation*}
Of course, $\rho_{\mathcal{A}} \leq \delta_{\mathcal{A}}$.

Now, while the metrics $\delta_{\mathcal{A}}$ and $\rho_{\mathcal{A}}$ depend on our choice of a finite atlas $\mathcal{A}$ of $M$, the induced topology does not --- if $\mathcal{B}$ is another finite atlas for $M$, then the change of basis formula between $(0,2)$ tensors immediately implies that $\delta_{\mathcal{A}}$ and $\delta_{\mathcal{B}}$ are equivalent, and similarly $\rho_{\mathcal{A}}$ and $\rho_{\mathcal{B}}$ are equivalent. That said, our assumption on $\mathcal{A}$ ensures that $\delta_{\mathcal{A}}$ is finite between any two Riemannian metrics.

\medskip

Now, a difficulty arises when working with Dirac operators for different Riemannian metrics: they act on different spinor bundles. We follow \cite{BG} to describe a way to obtain a unitary between the Hilbert spaces of volume-square-integrable sections of the spinor bundles for two different Riemannian metrics, in a canonical manner, so that this unitaries vary smoothly in the metric, in an appropriate sense which we now explain.

To begin with, fix $V = \R^d$. Let $g$ and $h$ be two inner products on $V$. Of course, there exists a unique positive linear endomorphism $H_g^h$ of $V$ such that $g(H_g^h\cdot,\cdot) = h(\cdot,\cdot)$. Moreover, the function $(g,h)\in\mathscr{I}(V)^2 \mapsto H_g^h$, where $\mathscr{I}(V)$ is the set of inner products on $V$, is smooth.

In the previous section, we use the Gram-Schmidt process to get a unitary $u$ from $(V,h)$ into $(V,g)$ such that $g(u\cdot,u\cdot) = h$; this unitary is not unique, though it is unique up to left multiplication by a unitary of $(V,g)$: if $u^\dagger$ is the adjoint of $u$ for $g$ (vs the adjoint $u^\ast$ of $u$ seen as a map from $(V,h)$ to $(V,g)$) then $u^\dagger u = H_g^h$, and if $o$ is any unitary of $(V,g)$, then $o u$ also satisfies $g(o u \cdot, o u \cdot) = h$. In the current setting, following \cite{BG}, we make a simpler choice which does not depend on the choice of local frames (which would involve local choices and require checking what happens in overlapping regions), and we instead define $b_g^h : V\rightarrow V$ by $b_g^h = \sqrt{H_g^h}$. By the functional calculus, $(g,h) \in \mathscr{I}(V)^2 \mapsto b_g^h$ is also smooth.

Now, we return to the manifold $M$ and fix $g \in \mathcal{R}(M)$. The function:
\begin{equation*}
	b_{g,m}^h : (h,m) \in \mathcal{R}(M)\times M \mapsto b_{g_m}^{h_m}
\end{equation*}
depends smoothly on the positive definite matrix entries and on the point in each chart. In turn, the map $b_g^h$ defined by setting, for all $X \in \Gamma_\infty(TM)$ (where $\Gamma_\infty(TM)$ is the $C^\infty(M)$-module of smooth sections of the tangent bundle $TM$ over $M$):
\begin{equation*}
	b_g^h (X) : m \in M \mapsto b_{g,m}^h(X_m)
\end{equation*}
is valued in $\Gamma_\infty(TM)$. Moreover, by construction, $b_g^h(f X) = f b_g^h(X)$ for all $f \in C^\infty(M)$ and for all $X \in \Gamma_\infty(TM)$, i.e. $b_g^h$ is a module morphism for $\Gamma_\infty(TM)$. We will abuse our notation a bit, and if $X \in TM_{|U}$, we define $b_g^h(X)$ as above, i.e. as a vector field over $U$ only.

As noted in \cite[Proposition 1]{BG}, the map $b_g^h$ is invertible with inverse $b^g_h$ for any two inner products $g,h$, which one can then easily extend to show that $b_g^h$ is invertible with inverse $b_h^g$ for any two $g,h \in \mathcal{R}(M)$.

Let $(U,x_1,\ldots,x_d)$ and let $e_1(g),\ldots, e_d(g)$ be the $g$-orthonormal frame of$TM_{|U}$ fixed in the first part of this section. We \emph{define}
\begin{equation*}
	e_j(h) \coloneqq b_h^g(e_j(g))
\end{equation*}
for $j\in\{1,\ldots,d\}$. By definition of $b_h^g$, the family $(e_1(h),\ldots,e_d(h))$ is a $h$-orthonormal frame of $TM_{|U}$. For each $j \in \{1,\ldots,d\}$ and $x\in U$, we define
\begin{equation*}
	a_{jk}(h,m) \coloneqq g_m (e_j(h)_m, e_k(g)_m) = [g(b_h^g(e_j(g)),e_k(g))](m)
\end{equation*}
so that
\begin{equation*}
	e_j^h = \sum_{k=1}^d a_{jk}(h) e_k(g) \text.
\end{equation*}
By continuity of $b_h^g$, we conclude that $a_{jk}$ is continuous on $\mathcal{R}(M)\times\closure{U}$ as well.

	The matrix square root and inversion are smooth on the positive definite matrices. On each compact $\closure{U}$, the eigenvalues of $g$ are bounded away from zero, and remain so for $h$ in a sufficiently small $C^0$ neighborhood of $g$. In a fixed local $g$-orthonormal frame, write $F(H)=H^{-1/2}$. The Bourguignon--Gauduchon formula and the chain rule give
	\begin{equation*}
		b_h^g=F(H_g^h),\qquad
		\partial_p b_h^g=DF(H_g^h)[\partial_p H_g^h].
	\end{equation*}
	Consequently $h_n\to h$ in $C^1$ implies $b_{h_n}^g\to b_h^g$ in $C^1$: both $F(H_g^{h_n})$ and $DF(H_g^{h_n})[\partial_pH_g^{h_n}]$ converge uniformly on each compact chart closure.
	Thus $e_j(h)\rightarrow e_j(g)$ in $C^1$ on every $\closure{U}$ as $h\rightarrow g$ in the $C^1$ topology. In particular, $a_{jk}(h,\cdot)\rightarrow\delta_{jk}$ uniformly.

\medskip

Now, by \cite[Proposition 5]{BG}, there exists a natural lift of the linear map  $b_{g_m}^{h_m}$, for any $m\in M$, to a linear isomorphism $\beta_g^h$ from $(\Sigma^h M)_m$ onto $(\Sigma^g M)_m$, which intertwines the actions of the Clifford algebra $c_{h_m}$ and $c_{g_m}$, i.e. $\beta_{g_m}^{h_m}c_{h_m}(X)=c_{g_m}(b_{g_m}^{h_m}X)\beta_{g_m}^{h_m}$ for $X\in T_mM$. Moreover, the map $m\in M \mapsto \beta_{g_m}^{h_m}$ is smooth, which in turn, allows us to define the $C^\infty(M)$-module map:
\begin{equation*}
	\forall \psi \in \Gamma_\infty(\Sigma^h M) \quad \beta_g^h(\psi) : m \in M \mapsto \beta_{g_m}^{h_m}(\psi_m)  \in \Gamma_\infty(\Sigma^g M) \text.
\end{equation*}
Once again whenever convenient, we will use the same notation for the map $\beta_g^h$ acting on $\Gamma_\infty(\Sigma^h M)_{|U}$ valued in $\Gamma_\infty(\Sigma^g M)_{|U}$. It is then straightforward to extend these maps by continuity to a $C(M)$-module morphism $\beta_g^h$ from $\Gamma_{L^2}(\Sigma^h M)$ to $\Gamma_{L^2}(\Sigma^g M)$. The resulting maps are linear isomorphisms, with inverse $\beta_h^g$.

However, the maps $\beta_g^h$, for $g,h \in \mathcal{R}(M)$, are not unitary because we also must account for the change in the volume form, as noted in \cite{Maier97}. We thus define:
\begin{equation}\label{fg-eq}
	f_g^h \coloneqq \left(\frac{\det(h)}{\det(g)}\right)^{1/4} \text.
\end{equation}
Note that $f_g^h$ is a smooth function and that $h \in \mathcal{R}(M) \mapsto f_g^h \in C(M)$ is a continuous function by the definition of the $C^0$-topology on $\mathcal{R}(M)$ --- for our purpose, this map is then of course continuous from $C^1$ topology on $\mathcal{R}(M)$ to the topology on $C(M)$ induced by the usual supremum norm, i.e. C*-norm, on $C(M)$.

We thus define, for all $\gamma \in \Gamma_{L^2}(\Sigma^h M)$,
\begin{equation*}
	\widehat{\beta_g^h} (\gamma) : m \in M \mapsto f_g^h(m) \beta_{g_m}^{h_m}(\gamma_m) \text,
\end{equation*}
i.e. the composition of the multiplication operator by $f_g^h$ with $\beta_g^h$.

The fiberwise isometry $\beta_g^h$ and the identity $(f_g^h)^2\,\mathrm{vol}_g=\mathrm{vol}_h$ give
	\begin{equation*}
		\norm{\widehat{\beta_g^h}\gamma}{L^2(g)}^2
		=\int_M(f_g^h)^2|\gamma_m|_h^2\,\mathrm{vol}_g
		=\norm{\gamma}{L^2(h)}^2.
	\end{equation*}
	Thus $\widehat{\beta_g^h}$ is a unitary from $\Gamma_{L^2}(\Sigma^h M)$ onto $\Gamma_{L^2}(\Sigma^g M)$ which intertwines the action of $C(M)$ by multiplication operators on each of these Hilbert spaces (i.e. it is a module morphism as well: $\widehat{\beta_g^h}(f \gamma) = f \widehat{\beta_g^h}(\gamma)$), since of course, the representations by multiplication operator of $C(M)$   on these Hilbert spaces commute with the multiplication by $f_g^h$, and $\beta_g^h$ was already a module map.

Moreover, by construction, for all $X\in TM$ and for all $\psi \in \Gamma_{L^2}(\Sigma^h M)$, we also have the following property, since multiplication by functions in $C(M)$ commutes with the Clifford multiplication:
\begin{equation*}
	\widehat{\beta_g^h}(c_h(X)\psi) = c_g(b_g^h(X)) \widehat{\beta_g^h}(\psi) \text.
\end{equation*}

\medskip

We are now ready to compare Dirac operators associated with different
Riemannian metrics. Fix a chart $(U,x^1,\ldots,x^d)$ and the spin
frame $s_g$ lifting the oriented $g$-orthonormal frame
$(e_1(g),\ldots,e_d(g))$ used in the spin connection calculation.
Let $s_h$ be its image under the Bourguignon--Gauduchon lift of
$b_h^g$. Then $s_h$ lifts $(e_1(h),\ldots,e_d(h))$, where
$e_j(h)=b_h^g(e_j(g))$. Both spin frames give coordinates in the
same model spinor space. In these coordinates, $\beta_g^h$ sends
a spinor with coordinates $v$ in $s_h$ to the spinor with the same
coordinates $v$ in $s_g$. Its matrix is therefore the identity
at every point of $U$.

Write $\partial_p=\partial/\partial x^p$ and
$C_j=c_g(e_j(g))$ in the chosen spin frame. The matrices $C_j$
are constant, and conjugation by $\widehat{\beta_g^h}$ sends
$c_h(e_j(h))$ to $C_j$. Since $\widehat{\beta_g^h}$ and its
inverse are represented in these coordinates by multiplication by
$f_g^h$ and $f_h^g=(f_g^h)^{-1}$, respectively, the product rule
gives, for every smooth spinor $\xi$,
\begin{align*}
 \widehat{\beta_g^h}\partial_p
       \bigl(\widehat{\beta_h^g}\xi\bigr)
 &= f_g^h\partial_p(f_h^g\xi)\\
 &= \partial_p\xi+f_g^h(\partial_p f_h^g)\xi.
\end{align*}
Using the local expression \eqref{Dg-eq} for each metric, we obtain
\begin{align*}
 (\Dirac_g^h-\Dirac_g)\xi
 &= \sum_{j=1}^d C_j\Bigg(
    \sum_{p=1}^d (e_j^p(h)-e_j^p(g))\partial_p\xi\\
 &\qquad+\sum_{p=1}^d e_j^p(h)f_g^h(\partial_p f_h^g)\xi\\
 &\qquad+\frac14\sum_{k,l=1}^d
    (\omega_{jkl}(h)-\omega_{jkl}(g))C_kC_l\xi
    \Bigg).
\end{align*}
The term involving $\partial_p f_h^g$ is the contribution from
the change of volume form. Indeed, $f_g^h\partial_p f_h^g=-\partial_p\log f_g^h$.
For the norm estimate, collect the terms of order zero in
\begin{equation*}
 B_U(h)=\sum_{j=1}^d C_j\left(
   -e_j(h)(\log f_g^h)
   +\frac14\sum_{k,l=1}^d
    (\omega_{jkl}(h)-\omega_{jkl}(g))C_kC_l
 \right).
\end{equation*}
Since each $C_j$ has operator norm one, we have
\begin{align*}
 \norm{(\Dirac_g^h-\Dirac_g)\xi}{L^2(U,g)}
 &\leq \sum_{j,p}
 \norm{e_j^p(h)-e_j^p(g)}{C(\closure U)}
 \norm{\partial_p\xi}{L^2(U,g)}\\
 &\quad+\norm{B_U(h)}{C(\closure U)}
          \norm{\xi}{L^2(U,g)},
\end{align*}
where the supremum norm of $B_U(h)$ is taken in the matrix
operator norm. Set
\begin{equation}\label{rg-eq}
 r_g(h)=\frac1{Q(g)}\sum_{U\in\mathcal A}
 \left(\sum_{j,p}\norm{e_j^p(h)-e_j^p(g)}{C(\closure U)}
       +\norm{B_U(h)}{C(\closure U)}\right).
\end{equation}
Summing the local estimates and applying
\eqref{equiv-H1-CDN-eq}, we obtain
\begin{equation}\label{D-D-norm-eq}
 \norm{(\Dirac_g^h-\Dirac_g)\xi}{L^2(g)}
 \leq Q(g)r_g(h)\norm{\xi}{H_1}
 \leq r_g(h)\CDN_g(\xi).
\end{equation}

	For each fixed smooth $h$, multiplication by the smooth invertible bundle map $\widehat{\beta_g^h}$ maps the first Sobolev space for $h$ onto that for $g$: in local frames this follows from the product rule for first derivatives and boundedness of the coefficients and their derivatives on $\closure U$. Thus all $\Dirac_g^h$ have domain $\Gamma_{H^1}(\Sigma^g M)$, and the displayed estimate extends from smooth spinors to this common domain.

	Expressions \eqref{Christoffel-eq} and \eqref{spin-eq} show that the coefficients $\omega_{jkl}(h)$ depend continuously on $h$ in the $C^1$ topology, with values in $C(\closure U)$. The same is true of the first derivatives of $\log f_g^h$. Since $f_g^g=1$ and $e_j(g)(\log f_g^g)=0$, we conclude that
	\begin{equation*}
		\lim_{h\rightarrow g}r_g(h)=0.
	\end{equation*}
	This proves the required continuity from the common graph domain to the Hilbert space.

	\begin{theorem}\label{main-thm-1}
		Let $M$ be a connected compact closed manifold with a fixed spin structure. Endow its space $\mathcal R(M)$ of smooth Riemannian metrics with the $C^1$ topology. For all $g\in\mathcal R(M)$,
		\begin{equation*}
			\lim_{h\rightarrow g}\spectralpropinquity{}
			((C(M),\Gamma_{L^2}(\Sigma^h M),\Dirac_h),
			(C(M),\Gamma_{L^2}(\Sigma^g M),\Dirac_g))=0.
		\end{equation*}
	\end{theorem}

	\begin{proof}
		We denote the commutator seminorm and graph norm for $g$ by $\Lip_g$ and $\CDN_g$. The spectral triples $(C(M),\Gamma_{L^2}(\Sigma^h M),\Dirac_h)$ and $(C(M),\Gamma_{L^2}(\Sigma^g M),\Dirac_g^h)$ are unitarily equivalent. Thus it suffices to apply Lemma (\ref{main-lemma}) on the latter, fixed Hilbert space.

		We have already proved that these operators have the common domain $\Gamma_{H^1}(\Sigma^g M)$ and satisfy \eqref{D-D-norm-eq}, with $r_g(h)\rightarrow0$. It remains to compare their commutator seminorms.

		For $f\in C^\infty(M)$, Clifford multiplication gives
		\begin{equation*}
			[\Dirac_g^h,f]
			=c_g(b_g^h\grad{h}{f})
			=\sum_j(e_j(h)f)C_j,
			\qquad [\Dirac_g,f]=\sum_j(e_j(g)f)C_j.
		\end{equation*}
		Set
		\begin{equation*}
			K_U=\max_{1\leq p\leq d}\sup_U\sum_{j,k}|g_{pk}e_j^k(g)|,
			\qquad y_{g,U}(h)=\sum_{j,p}\norm{e_j^p(h)-e_j^p(g)}{C(\closure U)},
		\end{equation*}
		and
		\begin{equation}\label{YK-eq}
			K=\sum_{U\in\mathcal A}K_U,\qquad Y_g(h)=\sum_{U\in\mathcal A}y_{g,U}(h).
		\end{equation}
		The identities
		\begin{equation*}
			\frac{\partial f}{\partial x^p}=\sum_{j,k}g_{pk}e_j^k(g)e_j(g)f,
			\qquad C_j[\Dirac_g,f]+[\Dirac_g,f]C_j=-2(e_j(g)f)1
		\end{equation*}
		imply $\norm{\partial_p f}{\infty,U}\leq K_U\Lip_g(f)$. Therefore
		\begin{equation*}
			\norm{[\Dirac_g^h-\Dirac_g,f]}{}
			\leq K Y_g(h)\Lip_g(f).
		\end{equation*}

		These estimates hold on the full commutator domain. Indeed, a Lipschitz function has bounded weak first derivatives in every chart, multiplication by it preserves $H^1$, and the product rule gives the displayed commutator formulas almost everywhere. Conversely, if multiplication by $f\in C(M)$ preserves $H^1$ and $[\Dirac_g,f]$ is bounded, its distributional expression in a spin frame is $\sum_j(e_j(g)f)C_j$. Anticommuting with $C_j$ shows that multiplication by each distribution $e_j(g)f$ is bounded on local $L^2$. Thus these distributions are $L^\infty$ functions. The coordinate identity above gives bounded weak first derivatives, so $f$ is Lipschitz. The same argument applies for $h$, proving that all commutator domains coincide with the Lipschitz functions on $M$.

		It follows that, on this common domain,
		\begin{equation*}
			|\Lip_h(f)-\Lip_g(f)|\leq K Y_g(h)\Lip_g(f),
			\qquad \lim_{h\rightarrow g}Y_g(h)=0.
		\end{equation*}
		Both hypotheses of Lemma (\ref{main-lemma}) hold, proving the theorem.
	\end{proof}

	\begin{corollary}
		If $h_n\rightarrow g$ in the $C^1$ topology, then the eigenvalues of $\Dirac_{h_n}$ in every compact interval with endpoints outside $\spectrum{\Dirac_g}$ converge, with multiplicity, as in Theorem (\ref{cont-thm}). Moreover,
		\begin{equation*}
			\spectrum{\Dirac_g}=
			\{\lim_{n\rightarrow\infty}\lambda_n:
			\lambda_n\in\spectrum{\Dirac_{h_n}}\text{ for all }n,
			\ (\lambda_n)_{n\in\N}\text{ converges}\}.
		\end{equation*}
	\end{corollary}
	\begin{proof}
		Apply Theorem (\ref{cont-thm}) and \cite[Theorem 5.2]{Latremoliere22} to Theorem (\ref{main-thm-1}).
	\end{proof}

\bigskip

We conclude by applying our work to so-called almost commutative models, meaning that we will take the product of an actual Dirac operator as above with a finite dimensional spectral triple. Our reason for doing so is two-fold: it presents intrinsic interest since these models are central to the work of Connes on the standard model, and it also proves that the methods  developed in this section extend from classical to noncommutative geometry.

We henceforth fix a metric spectral triple $(\B,\Hilbert[F],\Dirac[F])$ where $\Hilbert[F]$ is finite dimensional --- hence so is $\B$ since this spectral triple is metric.
The construction depends on the parity of the dimension of $M$, though ultimately our proof works similarly in all cases.

Assume first that $M$ is even dimensional. There exists a self-adjoint unitary operator $\theta$ on $\Gamma_{L^2}(\Sigma^g)$, commuting with the representation of $C(M)$ by multiplication and anticommuting with $\Dirac_g$, i.e. $\bigslant{\Z}{2}$ grading.
We define, on $\dom{\Dirac_g}\otimes\Hilbert[F]$:
\begin{equation*}
	\Dirac_{g,F} \coloneqq \Dirac_g \otimes 1 + \theta\otimes \Dirac[F]\text.
\end{equation*}

Now, if $M$ is odd dimensional, and there exists a grading operator $\theta$ on $\Hilbert[F]$ for $(\B,\Hilbert[F],\Dirac[F])$ (i.e. a self-adjoint unitary commuting with the action of $\B$ and anticommuting with $\Dirac[F]$), then we proceed similarly and define, on $\dom{\Dirac_g}\otimes\Hilbert[F]$:
\begin{equation*}
	\Dirac_{g,F} \coloneqq \Dirac_g\otimes\theta + 1 \otimes \Dirac[F]\text.
\end{equation*}
In both of these cases , the product spectral triple is defined by:
\begin{equation*}
	(C(M),\Gamma_{L^2}(\Sigma^g M), \Dirac_g)\times (\B,\Hilbert[F],\Dirac[F]) \coloneqq (C(M)\otimes \B, \Gamma_{L^2}(\Sigma^g M)\otimes \Hilbert[F], \Dirac_{g,F}) \text,
\end{equation*}
and we set $\Hilbert_{g} \coloneqq \Gamma_{L^2}(\Sigma^g M)\otimes \Hilbert[F]$. Note that the Hilbert space $\Hilbert_g$ depends on $g$ (as the spinor bundle does) but not on $\Dirac[F]$.

If both $M$ is odd dimensional and we have no grading for $(\B,\Hilbert[F],\Dirac[F])$, then our product is defined on a slightly different Hilbert space, namely:
\begin{equation*}
	(C(M),\Gamma_{L^2}(\Sigma^g M), \Dirac_g)\times (\B,\Hilbert[F],\Dirac[F]) \coloneqq (C(M)\otimes \B, \Gamma_{L^2}(\Sigma^g M)\otimes \Hilbert[F]\otimes\C^2, \Dirac_{g,F}) \text,
\end{equation*}
where:
\begin{equation*}
	\Dirac_{g,F} \coloneqq \Dirac_g\otimes 1\otimes \theta_1 + 1 \otimes \Dirac[F] \otimes \theta_2 \text,
\end{equation*}
with $\theta_1,\theta_2$ self-adjoint matrices acting on $\C^2$ and with the usual Clifford relation:
\begin{equation*}
	\forall j,k \in \{1,2\} \quad \theta_j \theta_k + \theta_k \theta_j = \begin{cases} 2 \text{ if $j = k$,} \\ 0 \text{ otherwise.} \end{cases}
\end{equation*}
We then write $\Hilbert_{g} \coloneqq \Gamma_{L^2}(\Sigma^g M)\otimes \Hilbert[F]\otimes\C^2$.

	We keep the choice of product fixed throughout a varying family. In particular, for the second construction we fix a grading $\theta$ on $\Hilbert[F]$ and restrict to finite Dirac operators anticommuting with it. Without such a fixed grading, we use the doubled construction for every member of the family.

	\begin{theorem}\label{main-thm-2}
		Let $M$ be a connected compact closed manifold with a fixed spin structure, let $\Hilbert[F]$ be finite dimensional, and let $\B$ act faithfully and unitally on $\Hilbert[F]$. Let $\mathcal I(\B,\Hilbert[F])$ be the space of self-adjoint operators $\Dirac[F]$ such that $(\B,\Hilbert[F],\Dirac[F])$ is a metric spectral triple, endowed with the operator norm topology. Fix one of the product constructions above, restricting $\mathcal I$ to operators anticommuting with a fixed grading in the second case. These products are metric spectral triples, and
		\begin{multline*}
			\lim_{(h,\Dirac[G])\rightarrow(g,\Dirac[F])}\spectralpropinquity{}
			\big((C(M),\Gamma_{L^2}(\Sigma^h M),\Dirac_h)\times(\B,\Hilbert[F],\Dirac[G]),\\
			(C(M),\Gamma_{L^2}(\Sigma^g M),\Dirac_g)\times(\B,\Hilbert[F],\Dirac[F])\big)=0,
		\end{multline*}
		where $\mathcal R(M)$ is endowed with the $C^1$ topology.
	\end{theorem}

	\begin{proof}
		We first treat even dimensional $M$. Write $F=\Dirac[F]$, $G=\Dirac[G]$, and transport the product for $h$ to $\Hilbert_g$ by $\widehat{\beta_g^h}\otimes1$. The bundle identification preserves the chirality operator, so the transported operator is
		\begin{equation*}
			\overline\Dirac_{h,G}=\Dirac_g^h\otimes1+\theta\otimes G.
		\end{equation*}
		All these operators have domain $\Gamma_{H^1}(\Sigma^g M)\otimes\Hilbert[F]$. They are self-adjoint with compact resolvent, since they are bounded self-adjoint perturbations of $\Dirac_g^h\otimes1$. Smooth $\B$-valued functions have bounded commutators and form a dense *-subalgebra of $C(M,\B)$.

		For $f\in C^\infty(M,\B)$, set
		\begin{equation*}
			A\coloneqq[\Dirac_g\otimes1,f],\quad B\coloneqq\theta\otimes[F,f],
			\qquad\Theta=\theta\otimes1\text{, so} \qquad C\coloneqq[\Dirac_{g,F},f]=A+B.
		\end{equation*}
		Since $\Theta A\Theta=-A$ and $\Theta B\Theta=B$, we have
		\begin{equation*}
			A=\tfrac12(C-\Theta C\Theta),\qquad
			B=\tfrac12(C+\Theta C\Theta),\qquad
			\max\{\norm{A}{},\norm{B}{}\}\leq\Lip_{g,F}(f),
		\end{equation*}
		where $\Lip_{g,F}(f)=\norm{[\Dirac_{g,F},f]}{}$.
		Locally $A=\sum_j C_j\otimes e_j(g)f$, and hence
		\begin{equation*}
			(C_j\otimes1)A+A(C_j\otimes1)=-2\,1\otimes e_j(g)f.
		\end{equation*}
		It follows that
		\begin{equation}\label{product-derivative-eq}
			\norm{e_j(g)f}{\infty,U}\leq\Lip_{g,F}(f),\qquad
			\norm{\partial_p f}{\infty,U}\leq K_U\Lip_{g,F}(f),\qquad
			\norm{[F,f]}{C(M,\mathcal B(\Hilbert[F]))}\leq\Lip_{g,F}(f).
		\end{equation}
		
Here $K_U$ is the constant used in \eqref{YK-eq}. We now show
that the full commutator domain consists of the $\B$-valued
Lipschitz functions. Suppose that multiplication by
$f\in C(M,\B)$ preserves $H^1$ and that $[\Dirac_{g,F},f]$
extends to a bounded operator on $L^2$. Applying $f$ to smooth
sections that are constant near a given point in a local
trivialization shows that each matrix coefficient of $f$ belongs
to $H^1_{\mathrm{loc}}$.

On smooth sections supported in a chart, the weak product rule
then gives the formula $[\Dirac_{g,F},f]=A+B$ above, with
$e_j(g)f$ interpreted as a weak derivative. The identity
$A=\tfrac12(C-\Theta C\Theta)$ shows that $A$ is bounded.
Taking its anticommutator with $C_j\otimes 1$ shows that
multiplication by $e_j(g)f$ is bounded on local $L^2$.
Thus $e_j(g)f$ is essentially bounded. The coordinate identity
used above now gives bounded weak coordinate derivatives of
$f$, so $f$ is Lipschitz.

Conversely, a $\B$-valued Lipschitz function has bounded weak
derivatives. The weak product rule shows that multiplication by
$f$ preserves $H^1$ and that its commutator is given by the same
bounded formula. Thus the commutator domain is independent of
$g$ and $F$.
		We verify metricity explicitly. Fix $x_0\in M$ and a state $\nu$ of $\B$, and put $\mu(f)=\nu(f(x_0))$. Let $d_F$ be the diameter of the state space of $\B$ for its commutator metric. For self-adjoint $f$, the last inequality in \eqref{product-derivative-eq} gives
		\begin{equation*}
			\norm{f(x_0)-\mu(f)1}{}\leq d_F\Lip_{g,F}(f).
		\end{equation*}
		The first inequality gives, along any piecewise smooth path $\gamma$,
		\begin{equation*}
			\norm{f(\gamma(1))-f(\gamma(0))}{}
			\leq\sqrt d\,\operatorname{length}_g(\gamma)\Lip_{g,F}(f).
		\end{equation*}
		For Lipschitz $f$ this follows as well from the weak derivative bound, by local mollification and passage to the uniform limit. Therefore the self-adjoint set
		$\{f:\Lip_{g,F}(f)\leq1,\ \mu(f)=0\}$ is uniformly bounded and equicontinuous. Since $\B$ is finite dimensional, it is totally bounded by the Arzel\`a--Ascoli theorem. The commutator seminorm is lower semicontinuous and satisfies the Leibniz inequality. Its kernel consists of scalar functions, by the derivative bound and metricity of the finite triple. The total boundedness criterion of \cite{Rieffel98a} proves metricity of the product.

		\medskip

		We next compare graph norms. Anticommutation with $\theta$ gives on smooth spinors, and then on the common domain,
		\begin{equation}\label{even-triple-eq}
			\norm{\Dirac_{g,F}\psi}{}^2
			=\norm{(\Dirac_g\otimes1)\psi}{}^2+
			\norm{(1\otimes F)\psi}{}^2.
		\end{equation}
		For an orthonormal basis $(u_j)$ of $\Hilbert[F]$, write $\psi=\sum_j\psi_j\otimes u_j$. By \eqref{D-D-norm-eq} and the triangle inequality in the Euclidean space of sequences,
		\begin{align*}
			\norm{((\Dirac_g^h-\Dirac_g)\otimes1)\psi}{}
			 & \leq r_g(h)\left(\sum_j(\norm{\psi_j}{}+\norm{\Dirac_g\psi_j}{})^2\right)^{1/2} \\
			 & \leq r_g(h)(\norm\psi{}+\norm{(\Dirac_g\otimes1)\psi}{})                        \\
			 & \leq r_g(h)\CDN_{g,F}(\psi).
		\end{align*}
		Consequently
		\begin{equation*}
			\norm{(\overline\Dirac_{h,G}-\Dirac_{g,F})\psi}{}
			\leq (r_g(h)+\norm{G-F}{\mathcal B(\Hilbert[F])})\CDN_{g,F}(\psi).
		\end{equation*}
		This verifies Condition (2) of Lemma (\ref{main-lemma}).

		On the common commutator domain, \eqref{product-derivative-eq} gives
		\begin{equation*}
			\norm{[(\Dirac_g^h-\Dirac_g)\otimes1,f]}{}
			\leq K Y_g(h)\Lip_{g,F}(f).
		\end{equation*}
		Let $d_{g,F}$ be the diameter of the state space of the product for its commutator metric. For self-adjoint $f$ and any fixed state $\mu$ of $C(M,\B)$,
		\begin{align*}
			\norm{[\theta\otimes(G-F),f]}{}
			 & \leq2\norm{G-F}{\mathcal B(\Hilbert[F])}\norm{f-\mu(f)1}{}    \\
			 & \leq2d_{g,F}\norm{G-F}{\mathcal B(\Hilbert[F])}\Lip_{g,F}(f).
		\end{align*}
		For a general $f$ in the commutator domain, apply this estimate to its real and imaginary parts. Their commutator seminorms are at most $\Lip_{g,F}(f)$, so we obtain on the whole domain
		\begin{equation*}
			|\Lip_{h,G}(f)-\Lip_{g,F}(f)|
			\leq\big(KY_g(h)+4d_{g,F}\norm{G-F}{\mathcal B(\Hilbert[F])}\big)\Lip_{g,F}(f).
		\end{equation*}
		The coefficient tends to zero as $(h,G)\rightarrow(g,F)$. Condition (1) holds, and Lemma (\ref{main-lemma}) proves continuity in the even dimensional case.

		\medskip

		For the second product, put $\Theta=1\otimes\theta$ and
		\begin{equation*}
			\overline\Dirac_{h,G}=\Dirac_g^h\otimes\theta+1\otimes G,
			\quad A=[\Dirac_g\otimes\theta,f],\quad B=1\otimes[F,f].
		\end{equation*}
		Now $\Theta A\Theta=A$ and $\Theta B\Theta=-B$, so the same two projections of $C=A+B$ bound both components by $\norm C{}$. Locally
		$A=\sum_j C_j\otimes\theta(e_j(g)f)$, and
		\begin{equation*}
			(C_j\otimes\theta)A+A(C_j\otimes\theta)=-2\,1\otimes e_j(g)f.
		\end{equation*}
		Thus \eqref{product-derivative-eq}, the common domain assertion and the metricity proof apply unchanged. Since $F\theta+\theta F=0$,
		\begin{equation*}
			\norm{\Dirac_{g,F}\psi}{}^2
			=\norm{(\Dirac_g\otimes\theta)\psi}{}^2+\norm{(1\otimes F)\psi}{}^2.
		\end{equation*}
		The graph comparison has coefficient $r_g(h)+\norm{G-F}{}$, and the commutator comparison has coefficient $KY_g(h)+4d_{g,F}\norm{G-F}{}$, exactly as above. This proves continuity for the fixed graded family.

		For the doubled product, put $\Theta=1\otimes1\otimes\theta_1$. Then
		\begin{equation*}
			\overline\Dirac_{h,G}=\Dirac_g^h\otimes1\otimes\theta_1+1\otimes G\otimes\theta_2,
		\end{equation*}
		and the commutator components are
		$A=\sum_j C_j\otimes(e_j(g)f)\otimes\theta_1$ and
		$B=1\otimes[F,f]\otimes\theta_2$.
		They satisfy $\Theta A\Theta=A$, $\Theta B\Theta=-B$, and
		\begin{equation*}
			(C_j\otimes1\otimes\theta_1)A+A(C_j\otimes1\otimes\theta_1)
			=-2\,1\otimes e_j(g)f\otimes1.
		\end{equation*}
		The two summands of $\Dirac_{g,F}$ anticommute, giving again the sum-of-squares identity \eqref{even-triple-eq}, with the corresponding tensor factors. Both $\theta_1$ and $\theta_2$ have norm one. Hence the domain, metricity, graph and commutator estimates just proved have precisely the same constants for this product as well. Lemma (\ref{main-lemma}) completes the proof.
	\end{proof}

	\begin{corollary}
		Under the hypotheses and fixed product convention of Theorem (\ref{main-thm-2}), let $h_n\rightarrow g$ in $C^1$ and $G_n\rightarrow F$ in operator norm. In every compact interval with endpoints outside $\spectrum{\Dirac_{g,F}}$, the eigenvalues of $\Dirac_{h_n,G_n}$ converge, counted with multiplicity, as in Theorem (\ref{cont-thm}). Moreover,
		\begin{equation*}
			\spectrum{\Dirac_{g,F}}=
			\{\lim_{n\rightarrow\infty}\lambda_n:
			\lambda_n\in\spectrum{\Dirac_{h_n,G_n}}\text{ for all }n,
			\ (\lambda_n)_{n\in\N}\text{ converges}\}.
		\end{equation*}
	\end{corollary}
	\begin{proof}
		Apply Theorem (\ref{cont-thm}) and \cite[Theorem 5.2]{Latremoliere22} to Theorem (\ref{main-thm-2}).
	\end{proof}

	In the spectral action formulation of almost commutative models
	\cite{Connes}, the dependence of the action on the metric data is
	expressed through the Dirac spectrum. In particular, if
	$f\in C_c(\R)$, choose $a<b$ outside $\spectrum{\Dirac_{g,F}}$ such
	that $\supp{f}\subseteq(a,b)$. By the preceding corollary, the ordered
	eigenvalues in $[a,b]$, repeated according to multiplicity, converge.
	Summing $f$ over these finitely many eigenvalues gives
	\begin{equation*}
		\operatorname{Tr}\big(f(\Dirac_{h_n,G_n})\big)
		\longrightarrow \operatorname{Tr}\big(f(\Dirac_{g,F})\big).
	\end{equation*}
	Thus spectral actions with continuous compactly supported cutoffs
	are continuous under the variations considered here.

\bibliographystyle{amsplain}
\bibliography{../thesis.bib}

\vfill
\end{document}